\documentclass[11pt,soneside,a4paper,reqno]{amsart}
\usepackage{pdfsync, enumerate}
\usepackage{stmaryrd}
\usepackage{mathrsfs}
\usepackage{multicol}
\usepackage{amsmath, amsthm, amscd, amssymb, latexsym, eucal}
\usepackage[all]{xy}
\def\serieslogo@{} \def\@setcopyright{} \makeatother

\usepackage{multienum}

\usepackage[colorinlistoftodos]{todonotes}

\usepackage{hyperref}
\usepackage{color}
\usepackage{cite}

\usepackage{tikz-cd}

\makeatletter
\renewcommand*\env@matrix[1][c]{\hskip -\arraycolsep
  \let\@ifnextchar\new@ifnextchar
  \array{*\c@MaxMatrixCols #1}}
\makeatother

\usepackage{color}

\numberwithin{equation}{section}
\newtheorem{thm}{Theorem}[subsection]
\newtheorem*{main-thm}{Main Theorem}
\newtheorem*{Auslander-thm}{Auslander's Theorem}

\newtheorem{cor}[thm]{Corollary}
\newtheorem{lem}[thm]{Lemma}
\newtheorem{prop}[thm]{Proposition}

\newtheorem{mainthm}{Theorem}

\theoremstyle{definition}
\newtheorem{defn}[thm]{Definition}
\newtheorem{rem}[thm]{Remark}
\newtheorem{exam}[thm]{Example}

\newtheorem{notation}[thm]{Notation}

\newtheorem{quest}[thm]{Question}

\newcommand{\A}{\mathcal{A}}

\newcommand{\E}{\mathcal E}

\newcommand{\I}{\mathcal I}
\newcommand{\J}{\mathcal J}
\newcommand{\K}{\mathcal K}

\renewcommand{\P}{\mathcal P}
\newcommand{\Q}{\mathcal Q}
\newcommand{\R}{\mathcal R}

\newcommand{\mS}{\mathcal S}

\newcommand{\mfp}{\mathfrak{p}}

\newcommand{\unit}{\mathfrak{e}}

\newcommand{\DOT}{\setlength{\unitlength}{1.2pt}\begin{picture}(2.5,2)
          (1,1)\put(2.5,2.5){\circle*{2}}\end{picture}} 
\newcommand{\bu}{\DOT}

\newcommand{\cohom}{\mathsf{H}^{\bu}}
\newcommand{\sfH}{\mathsf{H}}
\newcommand{\ccent}{\mathsf{C}^{\bu}}
\newcommand{\sfC}{\mathsf{C}}
\newcommand{\HH}{\mathsf{HH}^{\bu}}

\newcommand{\ldual}{{}^{\vee}}  
\newcommand{\rdual}{{}^{\vee} \!}  

\newcommand{\ot}{\otimes}

\newcommand{\psiL}{{}_{\psi}\Lambda_1}
\newcommand{\Lpsi}{{}_1\Lambda_{\psi}}

\DeclareMathOperator*{\opp}{\mathsf{op}}

\DeclareMathOperator*{\id}{\mathsf{id}}

  \DeclareMathOperator*{\maxx}{\mathsf{max}}
 
\DeclareMathOperator*{\Mod}{\mathsf{Mod}}

\DeclareMathOperator{\End}{\mathsf{End}}
\DeclareMathOperator{\lrp}{\mathsf{lrp}}
\DeclareMathOperator{\lrP}{\mathsf{lrP}}
\DeclareMathOperator*{\smod}{\mathsf{mod}}

\DeclareMathOperator*{\umod}{\underline{\mathsf{mod}}}
\DeclareMathOperator*{\uMod}{\underline{\mathsf{Mod}}}

\DeclareMathOperator*{\Proj}{\mathsf{Proj}}
\DeclareMathOperator*{\MaxSpec}{\mathsf{MaxSpec}}

\DeclareMathOperator{\Thick}{\mathsf{Thick}}

\DeclareMathOperator{\Hom}{\mathsf{Hom}}

\DeclareMathOperator{\VVec}{Vec}

\DeclareMathOperator{\Ext}{\mathsf{Ext}}

   \DeclareMathOperator*{\res}{\mathsf{res}}
   \DeclareMathOperator*{\ind}{\mathsf{ind}}

\DeclareMathOperator*{\For}{\mathsf{For}}

\DeclareMathOperator*{\GL}{\mathsf{GL}}

\DeclareMathOperator*{\Aut}{\mathsf{Aut}}

\newcommand{\env}{{\operatorname{\mathsf{env}}}}

\newcommand{\rad}{\mathsf{rad}}

\newcommand{\mtc}{\textup{M}$\Delta$\textup{C}}
\newcommand{\extto}{\xrightarrow}
\DeclareMathOperator{\cone}{\mathsf{cone}}
\DeclareMathOperator{\Spc}{\mathsf{Spc}}

\DeclareMathOperator{\Spech}{\mathsf{Spec}^{\mathsf{h}}}
\DeclareMathOperator{\ulrp}{\underline{\mathsf{lrp}}}
\DeclareMathOperator{\ulrP}{\underline{\mathsf{lrP}}}

\DeclareMathOperator{\Ann}{\mathsf{Ann}}
\DeclareMathOperator{\obj}{\mathsf{obj}}
\DeclareMathOperator{\Subsets}{\mathsf{Subsets}}

\DeclareMathOperator{\Loc}{\mathsf{Loc}}
\DeclareMathOperator{\Thom}{\mathsf{Thom}}
\DeclareMathOperator{\supp}{\mathsf{supp}}
\DeclareMathOperator{\ev}{\mathsf{ev}}
\DeclareMathOperator{\coev}{\mathsf{coev}}

\DeclareMathOperator{\Ideals}{\mathsf{Ideals}}
\DeclareMathOperator{\Submods}{\mathsf{Submods}}

\DeclareMathOperator{\tr}{\mathsf{tr}}

\newcommand{\comp}{\mathop{\raisebox{+.3ex}{\hbox{$\scriptstyle\circ$}}}}

\usepackage{latexsym}
\usepackage{pstricks}
\usepackage{comment}

\begin{document}

\title[Modules, bimodules, and unipotent Hopf algebras] 
{Noncommutative tensor triangular geometry:
Modules, bimodules, and unipotent Hopf algebras}

\author[Solberg]{\O yvind Solberg}
\address{\O yvind Solberg\\
Department of Mathematical Sciences\\
NTNU\\
N-7491 Trondheim, Norway }
\email{oyvind.solberg@ntnu.no}

\author[Vashaw]{Kent B. Vashaw}
\address{Kent B. Vashaw\\
  Department of Mathematics\\
  UCLA\\
  Los Angeles, USA}
\email{kentvashaw@math.ucla.edu}

\author[Witherspoon]{Sarah Witherspoon}
\address{Sarah Witherspoon\\
Mathematics Department\\
Texas A{\&}M, College Station, USA}
\email{sjw@tamu.edu}

\date{\today}
 


\begin{abstract}
  We initiate a program aimed at classifying thick ideals, Balmer spectra, and submodule categories of various stable categories of bimodules and modules for finite dimensional selfinjective algebras, and at clarifying the relationship between the universal Balmer support and the Hochschild cohomology support. In this paper, we focus mostly on the case of a unipotent Hopf algebra $A$. The stable category $\underline{\mathsf{lrp}}(A^{\mathsf{env}})$ of $A$-bimodules that are projective as left and as right $A$-modules is a monoidal triangulated category under $\otimes_A$, and acts naturally on the stable category $\underline{\mathsf{mod}}(A)$ of $A$. We show in this case that the Balmer spectrum $\mathsf{Spc}(\mathcal{E})$ of the thick subcategory ${\mathcal E}$ of $\underline{\mathsf{lrp}}(A^{\mathsf{env}})$ generated by $A$ is homeomorphic to $\mathsf{Spc}(\underline{\mathsf{mod}}(A))$ and defines an embedding $\mathsf{Spc}(\underline{\mathsf{mod}}(A)) \to \mathsf{Spc}(\underline{\mathsf{mod}}(A^{\mathsf{env}}))$. Subject to a conjectural description of spectra of finite tensor categories, we show that the spectrum of ${\mathcal E}$ is homeomorphic to ${\mathsf{Proj}}$ of the Hochschild cohomology ring of $A$, and that the Hochschild support coincides with the universal Balmer support. We show that any subcategory ${\mathcal K}$ of $\underline{\mathsf{lrp}}(A^{\mathsf{env}})$ containing a thick generator admits a surjective continuous map from $\mathsf{Spc}(\underline{\mathsf{mod}}(A))$. As a consequence, under the aforementioned conjecture, this spectrum is Noetherian, classifies the thick ideals of ${\mathcal K}$, and classifies thick ${\mathcal K}$-submodule categories of $\umod(A)$ via the Stevenson module-theoretic support. As examples, we present in detail the representations of finite $p$-groups.
\end{abstract}

\maketitle

\setcounter{tocdepth}{1} 

\section{Introduction}

The theory of support varieties in representation theory encompasses an array of techniques organized around a common principle: one wishes to assign, for each representation of a given group or algebra, a closed subset of some geometric or topological space in a way that reflects the homological and tensor-categorical properties of that representation. Support varieties come in many different forms, including the cohomological support varieties studied by Quillen \cite{Quillen1971}, Carlson's rank varieties for elementary abelian groups \cite{Carlson1983} and analogous algebras, 
the hypersurface support for complete intersections introduced by Avramov and Buchweitz \cite{AvramovBuchweitz2000}, the Hochschild support varieties for finite dimensional algebras introduced by Snashall and the first author \cite{SnashallSolberg2004}, and the $\pi$-point support of Friedlander and Pevtsova \cite{FriedlanderPevtsova2007} for finite group schemes.  Tensor triangular geometry was introduced by Balmer~\cite{Bal}
for symmetric tensor triangulated categories, unifying many of these
known settings. Balmer's universal support varieties for a symmetric tensor triangulated category $\K$ take values in a formally-defined topological space that is now called the Balmer spectrum of $\K$, denoted $\Spc( \K)$. For some of the best known classes of examples, such as representations of finite groups
and finite group schemes, these Balmer spectra and support spaces are homeomorphic
to cohomologically-defined spectra and support varieties.
This correspondence allows classification of tensor ideals in terms of well understood homological information. The classification of thick tensor ideals is a longstanding problem in many different contexts, which include stable homotopy theory, algebraic geometry, and representation theory \cite{Hopkins1987,NeemanChrom1992,Thomason1997,BCR}. Stevenson extended Balmer's support theory to triangulated categories admitting an action by a tensor triangulated category in \cite{Stevenson}, giving a version of support which in some cases classifies thick submodule categories.

Balmer's theory and subsequent developments by many authors are specifically 
designed for {\emph{symmetric}} monoidal triangulated categories,
where in particular the tensor product is commutative up to natural isomorphism.
It generalizes in a straightforward way to {\emph{braided}} monoidal triangulated categories.
When the tensor product is however {\emph{noncommutative}}, 
analogies to noncommutative rings highlight some potential obstacles. 
Modifications are needed to obtain an informative theory of spectra and support.
This task was undertaken in work of 
Buan, Krause, and the first author in~\cite{BKS2007},
and in work of Nakano, Yakimov, and the second author in a series of 
papers~\cite{NVY1,NVY2,NVY3,NVY4,NVY5}.
Again in some of the best known cases, for example modules for some quantum
groups and smash coproduct Hopf algebras, the noncommutative Balmer spectrum
and support spaces are homeomorphic to cohomologically-defined support varieties.

In this paper we aim at a less well understood setting 
in relation to this theory:
categories of finitely generated bimodules over a selfinjective algebra $\Lambda$,
with tensor product over $\Lambda$ itself,
and actions of categories of $\Lambda$-bimodules on categories of
finitely generated left $\Lambda$-modules via tensor product over $\Lambda$ again.
This tensor product on bimodules is highly noncommutative (see e.g.\cite{ACM2012,Schauenburg})
and it is quite different from that on many of the categories studied previously. 
For example, the noncommutativity for the quantum groups and smash coproduct
Hopf algebras in~\cite{NVY3} is manageable by way of
informative comparisons with braided categories or group actions that govern the
noncommutativity in a prescribed way.
In contrast, the less manageable noncommutativity in categories of
bimodules presents some complications.
Nonetheless we begin here the task of studying their stable categories 
and their actions via noncommutative Balmer spectra and supports, obtaining
some first results in general settings.
We restrict to the stable category of finitely generated $\Lambda$-bimodules 
that are projective as left and as right  $\Lambda$-modules 
(called {\emph{left-right projective $\Lambda$-bimodules}}), for which
the tensor product over $\Lambda$ is an exact bifunctor. 
This is a triangulated subcategory of the stable category $\umod(\Lambda^{\env})$ of the enveloping algebra $\Lambda^{\env}$ of $\Lambda$.
This stable left-right projective category, denoted $\ulrp(\Lambda^{\env})$, was originally defined by Auslander and Reiten in \cite{AuslanderReiten1991}.
It has been fundamental in the study of stable equivalences, a major focus in modular representation theory of groups and representation theory of finite dimensional algebras (see for instance~\cite{Rickard1991,Broue1992,Rouquier1995,Linckelmann1996,Pogorzaly2001,DugasMartinezVilla2007,Pogorzaly2011}). 
However, until now, its monoidal triangular geometry has not been studied. 

The goal of this paper is to initiate a program for systematically studying the monoidal triangular geometry, Balmer spectra, thick ideals, thick submodule categories, and support varieties for the bimodule categories $\ulrp(\Lambda^{\env})$. By results of \cite{NVY3} in the finite tensor category setting, we also expect the geometry of the monoidal triangulated subcategory $\E  \subseteq \ulrp(\Lambda^{\env})$ generated by $\Lambda$ to play an important role in this theory. Since the whole category $\ulrp(\Lambda^{\env})$ does not necessarily have a thick generator, we also expect that monoidal triangulated subcategories $\K \subseteq \ulrp(\Lambda^{\env})$ which admit some thick generator will be important test cases. 

We propose the following guiding questions, given a finite dimensional selfinjective algebra $\Lambda$, $\E = \Thick(\Lambda) \subseteq \ulrp(\Lambda^{\env})$, and $\K=\Thick(X)$ for some bimodule $X\in \ulrp(\Lambda^{\env})$ such that $\K$ is monoidal triangulated:
\begin{enumerate}
\item What are the Balmer spectra of $\ulrp(\Lambda^{\env})$, $\E$, and $\K$?
\item Does the Balmer spectrum of $\ulrp(\Lambda^{\env})$ classify its thick ideals?
\item Do the Balmer spectra of $\ulrp(\Lambda^{\env})$, $\E$, and $\K$ classify submodule categories of $\umod(\Lambda)$ by way of actions of these categories on $\umod(\Lambda)$ and the Stevenson support?
\item What is the precise relationship between the Hochschild cohomology supports on $\ulrp(\Lambda^{\env}),$ $\E$, $\K$, and $\umod(\Lambda)$, and the Balmer and Stevenson supports?
\end{enumerate}

We set about developing a theory to answer these questions in the following manner.

We begin in Section~\ref{sec:framework} by recalling definitions of
monoidal triangulated categories, their actions on other triangulated categories,
cohomological and Balmer spectra and support, and known connections among them.
We recall the Rickard idempotent functors utilized by Cai and the
second author in the noncommutative setting~\cite{CV}, generalizing
work of Stevenson for symmetric monoidal categories~\cite{Stevenson}. 
To define these idempotent functors, 
we must temporarily leave categories of finitely generated modules
and work instead in larger categories of compactly generated objects
(e.g.~all modules rather than just the finitely generated ones).
We use these idempotent functors later, in Section~\ref{sec:unipotent},
to define support for actions.

In Section~\ref{sec:lrp} we study the left-right projective category
of bimodules over a finite dimensional selfinjective algebra $\Lambda$ and its stable category.
We show that this is a rigid monoidal category, and so is the  
subcategory generated by the unit object $\Lambda$ itself.
We define some functors between module
and bimodule categories for later use specifically in the case
of a finite dimensional Hopf algebra.

Our main results are in Sections~\ref{sec:unipotent} and~\ref{sec:Hopf}.
We keep to a more general setting in Section~\ref{sec:unipotent}, that of an
action of a rigid monoidal triangulated category on another triangulated category, 
developing some theory for handling tensor ideals and subcategories
under additional conditions.
When both categories are monoidal, we give conditions under which
there is a bijection between their tensor ideals.
We define supports of objects under actions with additional
conditions. One interesting aspect of the setting of Section~\ref{sec:unipotent} is 
that for the applications we have in mind, we need to consider non-monoidal subcategories and non-monoidal functors: we impose unipotent assumptions at various points to get around
non-monoidal obstructions, so that these subcategories and functors can still be 
leveraged to understand the underlying geometry.

In Section~\ref{sec:Hopf}, we restrict to unipotent Hopf algebras for
the strongest results, and use the general theory of the preceeding section to prove:
\begin{mainthm}[See Theorem \ref{thm:spce}, Theorem \ref{thm:nvy-conj-gives-hochschild}, Theorem \ref{thm:spclrp}]
\label{mainthm}
Let $A$ be a finite dimensional unipotent Hopf algebra. 
\begin{enumerate}[\qquad \rm(a)]
\item There are homeomorphisms $\Spc (\E) \cong \Spc ( \umod(A)) \cong \supp(A)$, where $\supp(A)$ is the Balmer support of the bimodule $A$ considered as an object of $\umod(A^{\env})$. If $\Spc(\umod(A))$ is Noetherian, then thick $\E$-submodule categories of $\umod(A)$ are classified by $\Spc(\E)$ using the Stevenson support.
\item Assume \cite[Conjecture E]{NVY3}. Then $\Spc(\E)$ is Noetherian, and there is a homeomorphism $\Spc(\E) \cong \Proj (\HH(A))$ under which the Balmer and Stevenson supports are identified with the classical Hochschild cohomology support. In particular, the Hochschild cohomology support for objects of $\E$ satisfies a tensor product property.
\item Suppose $\Spc(\umod(A))$ is Noetherian. Then there is a surjective map 
\[\Spc (\umod(A)) \to \Spc (\ulrp(A^{\env}))
\] that factors through an orbit space $\Spc (\umod(A)) / \Aut(A)$. Furthermore, there is a bijection between thick ideals of $\ulrp(A^{\env})$ and thick $\ulrp(A^{\env})$-submodule categories of $\umod(A)$. 
\item Suppose $\Spc(\umod(A))$ is Noetherian, and that $\K$ is a rigid monoidal triangulated subcategory of $\ulrp(A^{\env})$ containing a thick generator. Then there is a surjective continuous map
\[
\Spc(\umod(A)) \to \Spc(\K),
\]
and as a consequence $\Spc(\K)$ is Noetherian and thick $\K$-submodule categories of $\umod(A)$ are classified by $\Spc(\K)$ using the Stevenson support.
\end{enumerate}
\end{mainthm}

In particular, given any finite subgroup $H$ of $\Aut(A)$, we can identify a particular monoidal triangulated subcategory of $\ulrp(A^{\env})$ which has a spectrum that is precisely the topological quotient $\Spc(\umod(A))/H$. This quotient identifies with $\Proj$ of the categorical center of the Hochschild cohomology ring for this category.

We summarize the various functors and maps on spectra that we produce in the course of proving Theorem \ref{mainthm}, in the case that $A$ is a finite-dimensional unipotent Hopf algebra such that $\Spc(\umod(A))$ is Noetherian, and $\K$ is an arbitrary monoidal triangulated subcategory of $\ulrp(A^{\env})$. We have the following functors, where monoidal functors are labeled by $\otimes$ and non-monoidal functors are labeled by $\neg \otimes$:

\vspace{\baselineskip}
\begin{center}
\begin{tikzcd}
\umod(A^{\env})                                                           &  &                                     \\
\ulrp(A^{\env}) \arrow[u, "\neg \otimes"] \arrow[rrd, "{G,\neg \otimes}"] &  &                                     \\
\K \arrow[u, "\otimes"]                                                   &  & \umod(A) \arrow[lld, "{F,\otimes}"] \\
\E \arrow[u, "\otimes"]                                                   &  &                                    
\end{tikzcd}
\end{center}
\vspace{\baselineskip}

We discuss the definition and properties of the functors $F$ and $G$ in Sections \ref{subsect:functorG} and \ref{subsect:functorF}. These functors induce maps on the level of spectra (these are the maps appearing in Theorem \ref{mainthm}):
\vspace{\baselineskip}
\begin{center}
\begin{tikzcd}
                          & \Spc(\umod(A^{\env}))         &  &                                                                                                       \\
\supp(A) \arrow[ru, hook] & \Spc(\ulrp(A^{\env}))         &  & \Spc(\umod(A))/\Aut(A) \arrow[ll, two heads]                                                          \\
                          & \Spc(\K)                      &  & \Spc(\umod(A)) \arrow[lld, "\cong"] \arrow[ll, two heads] \arrow[llu, two heads] \arrow[u, two heads] \\
                          & \Spc(\E) \arrow[luu, "\cong"] &  &                                                                                                      
\end{tikzcd}
\end{center}
\vspace{\baselineskip}
The map $\Spc(\umod(A)) \to \Spc(\K)$ is continuous if $\K$ has a thick generator, but might not be in general. The map $\Spc(\umod(A)) \to \Spc(\ulrp(A^{\env}))$ is not known to be continuous. 

Although we have a surjective map $\Spc(\umod(A)) \to \Spc(\ulrp(A^{\env}))$ in the case that $\Spc(\umod(A))$ is Noetherian, a full understanding of the points and the topology of $\Spc(\ulrp(A^{\env}))$ remains elusive. As a first step in this direction, we take a closer look at the class of examples given by finite $p$-groups
in characteristic~$p$,
proving that the Balmer spectrum of the left-right projective category has a unique closed point. 

\begin{mainthm}[See Theorem \ref{thm:spc-arbitrary-pgroup}]
\label{thm:second-mainthm}
Let $G$ be a finite $p$-group where $p$ is a prime, $k$ is an algebraically closed field of characteristic $p$, and $A=kG$ is the group algebra of $G$. Then $\Spc (\ulrp(A^{\env}))$ has a unique closed point. If $G$ is cyclic or generalized quaternion (for $p=2$), then $\Spc(\ulrp(A^{\env}))$ is the one-point topological space.
\end{mainthm}

This outcome is in fact anticipated in potential connections with
cohomologically-defined support, where the appropriate cohomology ring
is that of the center of the category~\cite{NVY3}, and categories of
bimodules generally have trivial center~\cite{Schauenburg,ACM2012}. We check that for an elementary abelian $p$-group, the categorical center of the cohomology ring, which is a natural candidate to use to understand $\Spc(\ulrp(A^{\env}))$, is trivial. 

All our algebras are $k$-algebras for a field $k$.
In the latter part of Section~\ref{sec:Hopf}, 
we will assume $k$ is algebraically closed  as indicated there.

\subsection*{Acknowledgements} The authors thank Paul Balmer and Greg Stevenson for helpful comments. 
The first author thanks the Department of Mathematical Science and the Faculty of Information Technology and Electrical Engineering at NTNU for support during his sabbatical at Texas A\&M University, and thanks the Department of Mathematics at Texas A\&M University for the great hospitality during his stay there in the autumn of 2024. 
The second author was partially supported by NSF postdoctoral fellowship DMS-2103272 and by an AMS-Simons Travel Grant. The second author is grateful for helpful conversations with Joshua Enwright and Peter Goetz. 
The third author was partially supported by NSF grant DMS-2001163. 

\section{The general framework}
\label{sec:framework}

This section is devoted to describing the general framework in which
we set our problem. It involves monoidal triangulated categories and
actions thereof on triangulated categories, Balmer spectra and
support, cohomological support, and thick ideal / thick submodule
correspondences. We start by recalling the general notion of a monoidal
triangulated category and reviewing its Balmer spectrum. 
We then recall the constructions
of various support theories for actions of monoidal triangulated
categories on triangulated categories: the cohomological support, the central cohomological
support, and the Stevenson support. Lastly, we indicate how these supports
are related to one another in the case of finite group algebras, and recall
the extent to which these correspondences are known hold for general monoidal triangulated categories.

\subsection{Monoidal triangulated categories}
Recall that a
monoidal triangulated category $\K$ is a triangulated category with
suspension functor $\Sigma$ from $\K$ to $\K$ and additional structure 
$\K = (\K, \Sigma, \otimes, \unit, \mathfrak{a}, \mathfrak{l},\mathfrak{r}, \lambda, \varrho)$
which we for short write \mtc. Here 
\begin{enumerate}[\rm(i)]
\item 
  $ -\otimes-\colon \K\times \K \to \K $
  is an additive bifunctor that is exact in each variable,
\item $\mathfrak{e}$ is an object in $\K$ called the \emph{tensor
    identity of $\K$}, 
\item $\mathfrak{a}\colon (-\otimes-)\otimes - \to -\otimes(-\otimes-)$ is
  a natural isomorphism, 
\item 
  $ \mathfrak{l}\colon \mathfrak{e}\otimes - \to 1_\K $
  and
  $ \mathfrak{r}\colon -\otimes\mathfrak{e}\to 1_\K $
  are natural isomorphisms, and 
\item 
  $\lambda\colon -\otimes \Sigma(-) \to \Sigma(-\otimes-)$
  \  and  \ 
  $\varrho\colon \Sigma(-)\otimes-\to \Sigma(-\otimes-)$
  are natural isomorphisms.
\end{enumerate}
A number of axioms must be satisfied (Anticommutativity
Axiom, Pentagon Axiom, Triangle Axiom);  we omit these here (see~\cite{BKSS,JK}). 

\subsection{Examples}
\label{sect:ex} The classic example of a \mtc\ is the stable
category $\umod (A)$ of finitely generated modules of a finite
dimensional Hopf algebra $A$ over a field $k$ with the tensor product
$\otimes = \otimes_k$. Another example is the stable left-right projective category 
$\ulrp (\Lambda^{\env})$ of a finite dimensional selfinjective algebra $\Lambda$; we recall the construction of this category now. Consider the stable category $\umod(\Lambda^{\env})$ of finitely generated $\Lambda$-bimodules. 
Recall that since $\Lambda$ is selfinjective, so is its enveloping algebra $\Lambda^{\env}$, by \cite[Proposition 11.5]{SkowronskiYamagata2011}, so that both $\umod(\Lambda)$ and $\umod(\Lambda^{\env})$ are triangulated categories \cite{Happel1988}. The forgetful functors $\smod(\Lambda^{\env}) \to \smod(\Lambda)$ and $\smod(\Lambda^{\env})\to \smod(\Lambda^{\opp})$ are exact and send projectives to projectives, and hence descend to triangulated functors
\[
{\For}_l: \umod(\Lambda^{\env}) \to \umod(\Lambda), \quad {\For}_r: \umod(\Lambda^{\env}) \to \umod(\Lambda^{\opp}).
\]
The kernel of any triangulated functor is a thick triangulated subcategory, and so $\ulrp(\Lambda^{\env}):=\ker(\For_l) \cap \ker(\For_r)$ is a thick triangulated subcategory of $\umod(\Lambda^{\env})$. This category $\ulrp(\Lambda^{\env})$ is the full subcategory of $\umod(\Lambda^{\env})$ corresponding to $\Lambda$-bimodules that are projective as left $\Lambda$-modules and as right $\Lambda$-modules. The tensor product $\otimes_{\Lambda}$ is well-defined on this stable level, and under this tensor product, $\ulrp(\Lambda^{\env})$ is a \mtc. 

\subsection{Ideals in a \mtc} In a \mtc\ we have the notions of thick subcategories and left,
right, two-sided and prime ideals, which we recall next.
\begin{defn}
Let $\K$ be a \mtc.
\begin{enumerate}[\rm(i)]
\item A \emph{thick subcategory} of $\K$ is a nonempty full triangulated subcategory closed under direct summands.
\item A \emph{thick} \textup{(1)} \emph{ideal}, \textup{(2)} \emph{right
    ideal}, or \textup{(3)} \emph{left ideal} $\I$
  in $\K$ is thick subcategory of $\K$ such that
  \begin{enumerate}[\rm(1)]
  \item for each $a\in \I$ and for all $b\in\K$, both objects $a\otimes b$ and
    $b\otimes a$ are in $\I$.
  \item for each $a\in \I$ and for all $b\in\K$, the object $a\otimes b$
    is in $\I$.
  \item for each $a\in \I$ and for all $b\in\K$, the object $b\otimes a$
    is in $\I$.
  \end{enumerate}
\item A proper thick ideal $\P$ of $\K$ is  \emph{prime} in $\K$ provided that
  for all thick ideals $\I$ and $\J$ in $\K$,
  if $\I\otimes \J\subseteq \P$, then $\I$ or $\J$ is contained in $\P$. 
\item A thick ideal $\P$ of $\K$ is \emph{completely prime} provided that
  for all objects $a,b\in \K$, 
  if $ a\otimes b\in \P$, then $a\in \P$ or $b\in \P$.
\item The category $\K$ is \emph{completely prime} provided that 
   all prime ideals in $\K$ are completely prime. 

\end{enumerate}
\end{defn}

\begin{notation}
When $\mS$ is a subcategory of $\K$, we denote the thick subcategory, left, right, and (two-sided) thick ideals generated by $\mS$, respectively, by $\Thick(\mS),$ $\langle \mS \rangle_l$, $\langle \mS\rangle_r$, and $\langle \mS \rangle$. That is, $\Thick(\mS), \langle \mS\rangle_l, \langle \mS\rangle_r,$ and $\langle \mS \rangle$ are respectively the  smallest thick subcategory, smallest left ideal, smallest right ideal, and smallest two-sided ideal containing $\mS$.
In some cases, where we need to clearly specify the category $\K$, we will write $\langle \mS \rangle_{\K}$ for $\langle \mS \rangle$. The collection of all (two-sided) thick ideals of $\K$ will be denoted $\Ideals(\K)$.
\end{notation}

We recall a few useful standard facts about thick ideals, for reference. 

\begin{lem}
\label{fin-gen}
Suppose $\I = \langle \mS \rangle$ for some collection of objects $\mS$ in a \mtc{} $\K$. If $x \in \I$, then there is a finite subset $\mS' \subseteq \mS$ such that $x \in \langle \mS' \rangle$. 
\end{lem}

\begin{proof}
Consider the collection of objects 
\[
\I' := \{ x \in \I \mid \exists \text{ a finite subset }\mS' \subseteq \mS \text{ such that }x \in \langle \mS' \rangle\}.
\]
The claim of the lemma is that $\I' = \I$. Clearly, $\mS \subseteq \I'$, so to show that $\I' = \I$ it suffices to show that $\I'$ is a thick ideal. This is straightforward to verify; for instance, if 
\[
x \to y \to z \to \Sigma x
\]
is a distinguished triangle with $x$ and $y$ in $\I'$, then there exist finite subsets $\mS'$ and $\mS''$ of $\mS$ such that $x \in \langle \mS'\rangle$ and $y \in \langle \mS''\rangle$; setting $\mS''':=\mS' \cup \mS''$, we have $z \in \langle \mS'''\rangle$. 
\end{proof}

\begin{lem}
\label{lem:thick-ideal-tens}
Let $\K$ be a {\mtc} and let $\I$ be the two-sided thick ideal $\langle x \rangle$ of $\K$ generated by an object $x \in \K$. If $y \in \I$, then there exist objects $x_1, x_2 \in \K$ such that $y \in \Thick(x_1 \otimes x \otimes x_2).$
\end{lem}

\begin{proof}

Consider the following subset of $\K$:
\[
\I' := \{ y \in \I \mid \exists\;\; x_1, x_2 \in \K \text{ with } y \in \Thick(x_1 \otimes x \otimes x_2)\}.
\]
We claim that $\I' = \I$, from which the proposition will follow. Note that $\I'$ is a thick subcategory. In particular, if
\[
y_1 \to y_2 \to y_3 \to \Sigma y_1
\]
is a distinguished triangle with $y_1$ and $y_2$ in $\I'$, then there exist $x_1,x_2, x_3,$ and $x_4$ such that $y_1 \in \Thick(x_1 \otimes x \otimes x_2)$ and $y_2 \in \Thick(x_3 \otimes x \otimes x_4)$. Then $y_1$ and $y_2$ are both in $\Thick((x_1 \oplus x_3) \otimes x \otimes (x_2 \oplus x_4))$, so $y_3 \in \Thick((x_1 \oplus x_3) \otimes x \otimes (x_2 \oplus x_4))$,
 hence $y_3 \in \I'$.  

Clearly $x$ is in $\I'$, since $x \in \Thick( \unit  \otimes x \otimes \unit)$. Since $\I$ is the thick ideal generated by $x$, i.e.~the smallest thick ideal containing $x$, to show that $\I=\I'$ it will suffice to show that $\I'$ is a thick ideal.
 Let $y \in \I'$ and $x_1, x_2 \in \K$ with $y \in \Thick(x_1 \otimes x \otimes x_2)$. Let $z \in \K$; we must show that $z \otimes y$ and $y \otimes z$ are in $\I'$. 
Since $y \in \Thick(x_1 \otimes x \otimes x_2)$, we have $z \otimes y \in \Thick(z \otimes (x_1 \otimes x \otimes x_2)) =  \Thick((z \otimes x_1) \otimes x \otimes x_2)$. 
Therefore, $z \otimes y \in \I'$; by the analogous argument $y \otimes z$ is also in $\I'$. Therefore $\I'$ is a thick ideal. 
\end{proof}

\subsection{The Balmer spectrum}\label{subsec:Balmerspec}
Prime ideals in a \mtc\ $\K$ comprise the Balmer
spectrum of $\K$, defined as follows; this is a topological space which serves as the 
setting for universal support varieties. 
These notions were first introduced in~\cite{Bal} and generalized to the 
noncommutative setting in \cite{BKS2007,NVY1}. 

\begin{defn}
Let $\K$ be a \mtc. The \emph{Balmer spectrum of $\K$} is the
topological space 
       \[\Spc(\K) = \{\P\subseteq \K\mid \P \textup{\ is a prime ideal in\ }
         \K\} \]
       with closed sets the arbitrary intersections of the
       base of closed sets given by all 
       \[\supp_\K(a) = \{\P\in \Spc(\K)\mid a\not\in \P\}\]
       for objects $a\in \K$. Then $\supp_\K$ may be taken to be a map
       from $\obj(\K)$ to closed sets in $\Spc(\K)$,
       called the \emph{Balmer support}.
\end{defn}

By definition, an arbitrary closed set of $\Spc(\K)$ has the form
\[
\supp_{\K}(\mS):=\{\P \in \Spc(\K) \mid \mS \cap \P = \varnothing\}
\]
for some collection $\mS$ of objects in $\K$.

We will omit the subscript and write $\supp$ for $\supp_{\K}$
when $\K$ is understood from context.

Recall that a subset of a topological space is called {\emph{Thomason}} if it is a union of closed subsets, each of which is the complement of a quasi-compact set. In particular, if the space is Noetherian, then Thomason subsets coincide with specialization-closed subsets, that is, unions of closed subsets. We will denote the collection of Thomason subsets of a topological space $X$ by $\Thom(X)$. 

A \mtc{} $\K$ is called rigid if every object has a left and a right dual; we recall this definition in more detail below in Section \ref{subsect:duals}. In this case, every thick ideal of $\K$ is equal to the intersection of prime ideals which contain it \cite[Proposition 4.1.1]{NVY2}, and ideals of $\K$ are often classified by Thomason subsets of $\Spc (\K)$. 
In particular, the sufficiency of conditions (a)--(d) of Theorem~\ref{thm:gen-idealsclass} below were verified, respectively, in \cite[Theorem 4.10]{Bal}, \cite[Corollary 22]{MallickRay2023}, \cite[Theorem 4.6]{Rowe}, and \cite[Theorem A.7.1]{NVY4}. 

\begin{thm}
\label{thm:gen-idealsclass}
Let $\K$ be a rigid \mtc\ that satisfies any of the following conditions:
\begin{enumerate}[\qquad \rm(a)]
\item $\K$ is braided;
\item $\K$ is completely prime;
\item $\Spc (\K)$ is Noetherian;
\item $\K$ has a thick generator.
\end{enumerate}
Then the maps
\begin{center}
\begin{tikzcd}
\I \arrow[rr, maps to]                &  & \bigcup_{x \in \I} \supp(x)           \\
\Ideals(\K) \arrow[rr, bend left]   &  & \Thom(\Spc (\K)) \arrow[ll, bend left] \\
\{x \in \K \mid \supp(x) \subseteq S\} &  & S \arrow[ll, maps to]                  
\end{tikzcd}\end{center}
give a bijection between thick ideals of $\K$ and Thomason subsets of $\Spc( \K)$.
\end{thm}

Classification of thick ideals by Thomason subsets does not come for free: there are examples of \mtc{s}\ not satisfying any of the conditions of Theorem \ref{thm:gen-idealsclass} for which thick ideals are not parametrized by Thomason subsets of the Balmer spectrum (see \cite[Example 7.7]{HuangVashaw2025}). The stable left-right projective category from Section \ref{sect:ex} in which we are interested is not known a priori to satisfy any of the conditions of Theorem \ref{thm:gen-idealsclass}, and indeed in examples can be shown explicitly not to satisfy (a) and (d) (see for instance Section \ref{subsect:cyclic} below). 
We will pursue this aspect further later.

There is a more general setup we have in mind for these categories beyond just classifying thick ideals, namely actions and submodule categories.

\subsection{Actions} We recall the notion of an action of a \mtc\
on a triangulated category; see~\cite{JK} for the full definition. 
Let $\mathcal{K} = (\K, \Sigma, \otimes, \unit, \mathfrak{a}, \mathfrak{l},\mathfrak{r},\lambda,\varrho)$ 
be a \mtc. An \emph{action}
  of $\K$ on a triangulated category
  $\mathcal{A}=(\mathcal{A},T)$ 
with suspension $T$ is given by the following:
\begin{enumerate}[\rm(i)]
\item a bifunctor
  $ -*-\colon \K\times \A\to \A $ 
  that is exact in each variable,
\item
  a natural isomorphism 
  $ \mathfrak{a}' \colon (-\otimes -)*- \to -*(-*-) $
  satisfying the Pentagon Axiom, 
\item a natural isomorphism 
  $ \mathfrak{l}'\colon\mathfrak{e}*- \to 1_\A $
  for which $\mathfrak{l}$,
  $\mathfrak{r}$, $\mathfrak{a}'$, $\mathfrak{l}'$ satisfy two
  Triangle Axioms, and
\item natural isomorphisms 
  \[\lambda'\colon -* T(-)\to T(-*-) \ 
  \text{ and } \ 
  \varrho'\colon \Sigma(-)*-\to T(-*-)  \]
 such that
\item the diagram
  \[\xymatrix{
      \mathfrak{e}*T(a) \ar[r]^{\mathfrak{l}'_{T(a)}} \ar[d]^{\lambda'}
        & T(a)\ar@{=}[d] \\
        T(\mathfrak{e}*a)\ar[r]^{T(\mathfrak{l}'_a)} & T(a)}\]
    is commutative for all $a$ in $\A$, and
\item the diagram
  \[\xymatrix{
\Sigma(x)*T(a)\ar[r]^{\varrho'_{x,T(a)}}
\ar[d]^{\lambda'_{\Sigma(x),a}} & T(x*T(a))\ar[d]^{T(\lambda'_{x,a})}\\
T(\Sigma(x)*a)\ar[r]^{T(\varrho'_{x,a})} & T^2(x*a)
      }\]
   is anti-commutative for all $x$ in $\K$ and $a$ in $\A$.
  \end{enumerate}
This notion of an action naturally leads to that of $\K$-submodules of $\A$ as we
recall next.

\begin{defn}
Let $\K$ be a \mtc\ acting on a triangulated category $\A$. A thick
subcategory $\mathcal{U}$ of $\A$ is called a \emph{$\K$-submodule of
  $\A$} if $\K*\mathcal{U} \subseteq \mathcal{U}$. 
Denote the collection of all thick $\K$-submodule categories of $\A$ by $\Submods_{\K}(\A)$. 
\end{defn}

We turn to a generalization of the problem raised at the end of
Section \ref{subsec:Balmerspec}. Namely, is there a correspondence
between $\K$-submodules of $\A$ and Thomason subsets of some subset of $\Spc (\K)$? To make this question precise, we define notions of support
based on actions next.

\subsection{The cohomological support of an action} 
\label{sect:cohomsupport}
We recall graded homomorphisms and endomorphisms  and the support varieties defined from them.
For complete definitions and theory, see for instance~\cite{BIK2008,BKSS}. We continue our assumption from the previous section that $\K$ is a \mtc\ acting on a triangulated category $\A$.

We define the graded endomorphism ring of $\K$ by
\[  
  \cohom(\K)   := \End_{\K}^{\bu}(\mathfrak{e})
  = \oplus_{i\geq 0} \Hom_{\K}(\mathfrak{e}, \Sigma^i(\mathfrak{e})),
\] 
which is known to be  graded commutative (see, for example,~\cite{SuarezAlvarez04}). 
When $\K$ is understood from context, we will abbreviate $\cohom(\K)$ to just $\cohom$.

For any object $a$ in $\A$ there is a homomorphism of graded rings
\[\gamma_a\colon \cohom\to \End_\A^{\bu}(a)=
    \oplus_{i\geq 0} \Hom_{\A} (a, T^i(a))\]
given as follows. Let $g\colon \mathfrak{e}\to
\Sigma^i(\mathfrak{e})$ and define $\gamma_a(g)$ to be the composition  
\[  a \extto{\mathfrak{l}'^{-1}} \mathfrak{e}* a
  \extto{g * 1_{a}} \Sigma^i(\mathfrak{e})* a
  \extto{\rho'_i} T^i(\mathfrak{e} * a)
  \extto{T^i(\mathfrak{l}')} T^i(a).\]
For any two objects $a,b$ in $\A$, the two ring homomorphisms
$ \gamma_a\colon \cohom\to \End_\K^{\bu}(a)$
and
$ \gamma_b\colon \cohom \to \End_\K^{\bu}(b)$ 
give rise to two actions of $\cohom$ on
\[\Hom_\A^{\bu}(a,b) = \oplus_{i\geq 0}\Hom_\A(a, \Sigma^i(b))\]
as follows. Let $g\colon \mathfrak{e}\to
\Sigma^i(\mathfrak{e})$ and $f\colon a\to \Sigma^j(b)$ and define 
\[g\cdot f = \Sigma^j(\gamma_b(g))f \ 
\text{ and } \ 
f\cdot g = \Sigma^i(f)\gamma_a(g).\]
It can be shown that 
$g\cdot f = (-1)^{ij}f\cdot g$. 

The \emph{cohomological support variety} $V_{\sfH}(a)$ of an object $a$ in $\A$ is 
\[V_{\sfH}(a) = \{ \mathfrak{p}\in \Proj(\cohom)\mid
  \Ann_{\cohom}(\End^{\bu}_{\A}(a))\subseteq \mathfrak{p}\},\]
where $\Proj(\cohom)$ is the topological
space consisting of all non-irrelevant graded prime ideals in $\cohom$ with closed
sets given by all 
\[Z(I) = \{\mathfrak{p}\subseteq \Proj(\cohom) \mid I \subseteq \mathfrak{p}\}\]
for graded ideals $I$ in $\cohom$.

\begin{rem}
One can also define the cohomological support varieties using the homogeneous spectrum $\Spech(\cohom)$ of $\cohom$ rather than $\Proj (\cohom)$. Classically, the analogous support varieties defined on $\MaxSpec (\cohom)$ rather than $\Proj(\cohom)$ or $\Spech(\cohom)$ are also sometimes used, see for instance \cite{AvruninScott1982}. However, from the perspective of tensor triangular geometry and thick ideal classification, either $\Proj$ or $\Spech$ is likely closer to the Balmer spectrum than $\MaxSpec$, since in the Balmer spectrum irreducible closed sets have generic points \cite[Proposition 2.18]{Bal}.
\end{rem}

Since $\K$ acts on itself via its tensor product, 
we also obtain the support variety of
an object $x$ in $\K$ as
\[V_{\sfH}(x) = \{ \mathfrak{p}\in \Spech(\cohom)\mid
  \Ann_{\cohom}(\End^{\bu}_{\K}(x))\subseteq \mathfrak{p}\}.\]

\begin{exam}
Let $\K$ be a finite tensor category.
The cohomological support with respect to the action of the stable category of $\K$ on itself
is precisely the classical cohomological support discussed for example in \cite{BPW2021}. 
On the other hand, if $\Lambda$ is a finite dimensional selfinjective algebra,
then the cohomological support with respect to the action 
of the stable left-right projective category $\ulrp(\Lambda^{\env})$ on itself or on $\umod(\Lambda)$
is precisely the Hochschild cohomology support as in \cite{SnashallSolberg2004,Solberg2006}.
\end{exam}

\subsection{The categorical center and the comparison map}
\label{subsect:catcenter}

For a general \mtc{} $\K$, we now have two versions of support: the Balmer support $\supp_{\K}$, and the cohomological support $V_{\sfH}$. In order to compare these two supports, we recall the categorical center, a particular subalgebra of $\cohom(\K)$. 
There is in general a comparison map $\rho$ from the Balmer spectrum $\Spc (\K)$ to the homogeneous spectrum of the categorical center $\ccent(\K) \subseteq \cohom(\K)$, by \cite[Theorem B]{NVY3}, which we also recall now. 

The {\em categorical center} $\ccent(\K)$, with respect to some fixed set $\mS$ of thick generators for $\K$, is defined to be all $g \in \cohom(\K)$ such that the diagram
\begin{center}
\begin{tikzcd}
\unit \otimes x \arrow[d, "f \otimes \id_x"] \arrow[r, "\cong"] & x          & x \otimes \unit \arrow[l, "\cong"'] \arrow[d, "\id_x \otimes f"] \\
\Sigma^i \unit \otimes x \arrow[r, "\cong"]                     & \Sigma^i x & x \otimes \Sigma^i \unit \arrow[l, "\cong"']                    
\end{tikzcd}
\end{center}
commutes for all $x$ in $\mS$. When $\K$ is thickly generated by its tensor unit $\unit$, then the categorical center relative to the generating set $\{\unit\}$ is the whole cohomology ring $\cohom$. When $\K$ and $\mS$ are understood from context, we will abbreviate $\ccent(\K)$ by $\ccent$. The comparison map $\rho$ mentioned above is explicitly defined via
\begin{align*}
\rho \colon \Spc (\K) & \to \Spech(\ccent )\\
\P & \mapsto \langle f \in \ccent (\K) \mid f \text{ is homogeneous, and }\cone(f) \not \in \P\rangle.
\end{align*}
If $\K$ is a  finite tensor category, then contingent on the Etingof--Ostrik 
cohomological finite generation conjecture \cite{EO2004}, one can use the correspondence between complexity of an object and dimension of its support variety \cite[Corollary 4.2]{BPW2021} to show that $\rho$ induces a surjective map $\Spc (\K) \to \Proj (\ccent)$ \cite[Corollary 7.1.3, Theorem 7.2.1(b)]{NVY3}. 
If $\K$ is braided, the categorical center is the whole cohomology ring $\cohom$, and the comparison map $\rho$ was introduced by Balmer in \cite{Balmer2010}. 

We now replace $\cohom$ in the cohomological support $V_{\sfH}$ by the
categorical center $\ccent$, to define the central cohomological support variety for any $x \in \K$:
\[
   V_{\sfC}(x) = \{ \mathfrak{p} \in \Proj (\ccent) \mid
    \Ann_{\ccent}(\End_{\K}(x))\subseteq \mathfrak{p}\} .
\]
Under the finite generation condition mentioned above, the comparison map $\rho$ is a closed map, and 
\[
\rho(\supp_{\K}(M)) = V_{\sfC}(M)
\]
for all $M$ in $\K$; see \cite[Proposition 8.4]{CV}. It is conjectured in \cite[Conjecture E]{NVY3} that for all finite tensor categories, $\rho$ is a homeomorphism. 

\begin{rem}
Note that one can construct similar supports and comparison maps substituting other invertible objects besides $\Sigma( \unit)$, as is done in \cite[Section 3]{Balmer2010}, which is useful in particular in algebraic geometry (see \cite[Remark 8.2]{Balmer2010}). 
\end{rem}

\subsection{Rickard idempotents and the Stevenson support for an action}
\label{subsect:RickStevenson}

Assume now that $\K$ is the compact part of a rigidly-compactly generated \mtc{}. That is, assume there is a \mtc{} $\widehat{\K}$ that contains all set-indexed coproducts, whose rigid objects coincide with the compact objects, and having subcategory of compact objects $\widehat{\K}^c \cong \K$.

In addition to the cohomological support of an action recalled above, 
there is a notion of support for an action of a \mtc\  $\K$ on $\A$ which 
takes values in the Balmer spectrum of $\K$, defined in the symmetric case 
by Stevenson \cite{Stevenson}. This support is based on the Rickard idempotent 
functors (as defined originally by Neeman \cite{Neeman1992} and Rickard 
\cite{Rickard1997}). We recall these idempotent functors now. 

\begin{thm}[Rickard~\cite{Rickard1997}]
\label{thm:Rickard}
Let $\I$ be a thick subcategory of $\K$. There exist functors 
${\it{\Gamma}}_{\I}$ and $L_{\I}\colon \widehat{\K} \to \widehat{\K}$ such that
\begin{enumerate}[\qquad \rm(a)]
    \item $\it{\Gamma}_{\I}$ takes values in $\Loc(\I)$, the localizing subcategory of $\widehat{\K}$ generated by $\I$;
    \item $L_{\I}$ takes values in $\Loc(\I)^{\perp}$, that is, the subcategory of objects admitting no maps from $\Loc(\I)$; 
    \item for any object $A \in \widehat{\K}$, there is a unique distinguished triangle
\[
\it{\Gamma}_{\I} A \to A \to L_{\I} A.
\]
\end{enumerate}
\end{thm}

The functors ${\it{\Gamma}}_{\I}$ and $L_{\I}$ are called the {\it{Rickard idempotent functors}}. 
We recall some of the basic properties of these functors for 
reference (cf.~ \cite[Section 3]{NVY4}).

\begin{thm}
\label{thm:Rickard-props}
Let $\K$ be the compact part of a compactly-generated monoidal triangulated 
category $\widehat{\K}$, and let $\I$ be a thick subcategory of $\K$. 
The Rickard idempotent functors satisfy:
\begin{enumerate}[\qquad \rm(a)]
\item $x \in \K$ is in $\Loc(\I)$ if and only if $L_{\I}(x) \cong 0$;
\item if $\I$ is a right ideal of $\K$, then $L_{\I} x \cong L_{\I} \unit \otimes x$ 
and ${\it{\Gamma}}_{\I} x \cong {\it{\Gamma}}_{\I} \unit \otimes x$, for all $x \in \K$;
\item if $\I$ is a left ideal of $\K$, then $L_{\I} x \cong x \otimes L_{\I} \unit$ 
and ${\it{\Gamma}}_{\I} x \cong x \otimes {\it{\Gamma}}_{\I} \unit$ for all $x \in \K$;
\item if $\I$ is either a left or right ideal, then 
$L_{\I} \unit \otimes L_{\I} \unit \cong L_{\I} \unit$ 
and ${\it{\Gamma}}_{\I} \unit \otimes {\it{\Gamma}}_{\I} \unit \cong {\it{\Gamma}}_{\I} \unit$.
\end{enumerate}
\end{thm}

Using Rickard idempotent functors, Balmer and Favi developed in \cite{BF} a 
version of support for non-compact objects in a symmetric tensor triangulated 
category $\K$ based on $\Spc (\K)$, parallel to the big support developed 
by Benson--Iyengar--Krause in \cite{BIK2008}. This theory was fully generalized 
to the noncommutative case by Cai and the second author in \cite{CV}. 
In the symmetric case, Stevenson extended this support to actions \cite{Stevenson}. 
This generalizes immediately to the noncommutative 
case, as we now record. We restrict to the case that $\Spc (\K)$ is Noetherian for simplicity. 

First, some notation. Let $\P \in \Spc (\K)$, that is, $\P$ is a prime thick ideal in 
$\K$. We define two specialization-closed subsets of $\Spc (\K)$ corresponding to 
$\P$ (cf. \cite[Definition 3.21]{Stevenson2018}): 
\begin{align*}
    \mathcal{V}(\P) &:=\overline{\{\P\}} = \{\Q \in \Spc (\K) \mid \Q \subseteq \P\},\\
    \mathcal{Z}(\P) &:= \{\Q \in \Spc (\K) \mid \P \not \in \mathcal{V}(\Q)\} = \{ \Q \in \Spc (\K) \mid \P \not \subseteq \Q\}.
\end{align*}
Now we set
\[
{\it{\Gamma}}_{\P} := {\it{\Gamma}}_{\mathcal{V}(\P)} \circ L_{\mathcal{Z}(\P)}.
\]
We emphasize that whenever we write ${\it{\Gamma}}_{\P}$ for $\P \in \Spc (\K)$, we use the above definition, rather than the other potential meaning for ${\it{\Gamma}}_{\P}$ which comes from Theorem \ref{thm:Rickard}. Note that by Theorem \ref{thm:Rickard-props}(b), (c) the functor ${\it{\Gamma}}_{\P}$ is given by the tensor product with ${\it{\Gamma}}_{\mathcal{V}(\P)} \unit \otimes L_{\mathcal{Z}(\P)} \unit \cong L_{\mathcal{Z}(\P)} \unit \otimes {\it{\Gamma}}_{\mathcal{V}(\P)} \unit$.

Finally, we define the support of an object $a \in \A$.

\begin{defn}
Let $a \in \A$. The {\emph{Stevenson module-theoretic support of $a$}} is 
\[
\supp^*(a):=\{\P \in \Spc (\K) \mid {\it{\Gamma}}_{\P} \unit * a \not = 0\} \subseteq \Spc (\K).
\]
\end{defn}

\begin{exam}
    By \cite[Proposition 3.8]{CV}, if we take $\A=\K$ and the tensor action to be the canonical one, then $\supp^*(a) = \supp(a).$
\end{exam}

\begin{rem}
Part of the usefulness of defining support in terms of Rickard idempotents is that it gives a coherent notion of support for non-compact objects. However, in this paper we will focus on applications to compact objects, so we do not emphasize the non-compact part of the story here. 
\end{rem}

The proof of the next proposition is standard.
For example, one of the containments in part (d),
that $\supp^*(x * a) \subseteq \supp^*(a)$, follows from 
properties of Rickard idempotents (see Theorem~\ref{thm:Rickard-props}(b) and (c)). 

\begin{prop}
    \label{prop:support-gamma-props}
    The module-theoretic support $\supp^*$ on $\A$ based in $\Spc (\K)$ satisfies 
the following properties:
    \begin{enumerate}[\qquad \rm(a)]
        \item $\supp^*(0) = \varnothing$;
        \item $\supp^*(\Sigma a) =\supp^*(a)$ for any $a \in \A$;
        \item if $a \to b \to c \to \Sigma a$ is a distinguished triangle 
in $\A$, then 
  \[ \supp^*(a) \subseteq \supp^*(b) \cup \supp^*(c);\]
        \item $\supp^*(x*a)\subseteq \supp (x) \cap \supp^*(a)$ for all 
$x \in \K$ and $a \in \A$.
    \end{enumerate}
\end{prop}

Using module-theoretic support, we now have two maps: 
\begin{center}
\begin{tikzcd}
\I \arrow[rr, maps to, "\Phi"]                &  & \bigcup_{a \in \I} \supp^*(a)           \\
\Submods_{\K}(\A) \arrow[rr, bend left]   &  & \Subsets(\Spc (\K)) \arrow[ll, bend left] \\
\{a \in \A \mid \supp^*(a) \subseteq S\} &  & S \arrow[ll, maps to, "\Theta"]                  
\end{tikzcd}\end{center}
where $\Subsets(\Spc (\K))$ is the collection of all subsets of $\Spc (\K)$.

Now suppose $\K$ is symmetric and 
satisfies the {\em{local-to-global principle}}, that is, for all
$a \in \widehat{\K},$ we have
\[
\Loc^{\otimes}(a) = \Loc^{\otimes}({\it{\Gamma}}_{\P} a \mid \P \in \Spc (\K)),
\]
where $\Loc^{\otimes}(a)$ is the smallest localizing tensor ideal of $\widehat{\K}$ containing $a$.
Stevenson proved in \cite[Lemma 4.12, Proposition 4.13]{Stevenson2018}
that the map $\Theta$ defined above is injective when restricted to $\Phi(\A)$. 
Note that as long as $\Spc (\K)$ is Noetherian, $\K$ satisfies the local-to-global principle (see \cite{Stevenson}, and \cite{CV} for the noncommutative case). 

\begin{defn}
\label{defn:param}
The action of $\K$ on $\A$ is {\em{parametrizing}} if $\Theta$ and $\Phi$ defined above give a bijection between submodule categories of $\A$ and Thomason subsets of $\Phi(\A)$. 
\end{defn}

When we take the action of $\K$ on itself by its tensor product, then as long as $\K$ is braided, completely prime, $\Spc (\K)$ is Noetherian, or $\K$ is thickly generated by a single object, then the action is parametrizing by Theorem \ref{thm:gen-idealsclass}. Stevenson proved in \cite[Theorem 6.13]{Stevenson2014} that when $R$ is a commutative Noetherian ring which is locally a hypersurface, then the action of the unbounded derived category of $R$ on the singularity category or stable derived category of $R$ (see \cite{Buchweitz, Orlov}) is parametrizing. Part of our motivation, moving forward, is to determine whether or not the naturally defined actions of $\ulrp(\Lambda^{\env})$ and its various subcategories on $\umod(\Lambda)$ are parametrizing.

\begin{rem}
We should emphasize again that the Stevenson module-theoretic support developed in \cite{Stevenson}, and the notion of parametrizing actions, were actually built not just to handle thick submodule categories, but in fact to study the lattice of localizing submodule categories. As remarked above, we will focus on the compact case in this paper, so we omit discussing this broader context here.
\end{rem}

\subsection{Possible correspondences}\label{subsec:correspondences}
Using all the general notions introduced above, we can now formulate more precisely the
problems and the objects we study.  For a given \mtc\ $\K$
acting on triangulated category $\A$, there are at least four objects of interest:
\begin{enumerate}[\rm(i)]
\item the Balmer spectrum $\Spc(\K)$ of $\K$,
\item the (left/right/two-sided) ideals in $\K$,
\item the $\K$-submodules of $\A$,
\item the spectrum of the graded endomorphism ring $\cohom = \End_\K^{\bu}(\mathfrak{e})$
  of the tensor identity $\mathfrak{e}$
  and that of its categorical center $\ccent$ with respect to a thick generating set of $\K$.
\end{enumerate}
The aim of the paper is to show how all these four objects are
related, and more. A guiding example is representations of finite
groups. In particular, the results are especially nice for
$p$-groups. If $G$ is a $p$-group and $k$ is a field of characteristic $p$, 
then there is a one-to-one
correspondence between thick subcategories of $\umod (kG)$ and
nonempty sets of closed homogenous subvarieties of $\Proj (\cohom(G,k))$
closed under specialization and finite unions (see \cite{BCR}). 

We know in general that (i) and (ii) above are related via Theorem \ref{thm:gen-idealsclass} as long as $\K$ is braided, has a Noetherian spectrum, has a thick generator, or is completely prime; however, these conditions do not hold for some examples we wish to study. Correspondences between (i) and (iii) have been extensively studied by Stevenson \cite{Stevenson} in the symmetric case, but up to now have not been studied in the noncommutative case. There is a comparison map between (i) and (iv), recalled above in Section \ref{subsect:catcenter}; this comparison map is a homeomorphism for many (but far from all!) \mtc{s}. It is conjectured always to be a homeomorphism for stable categories of finite tensor categories, but the noncommutative comparison map has, up to now, not been studied in examples outside of the finite tensor category setting.

\section{The stable left-right projective category of bimodules}
\label{sec:lrp}

Our prime examples of \mtc{s}\ are $\umod (A)$
and $\ulrp (\Lambda^{\env})$ for a finite dimensional Hopf algebra $A$ and
selfinjective algebra $\Lambda$ (as explained in Section \ref{sect:ex}). 
As we saw above, in particular in relation to Theorem \ref{thm:gen-idealsclass}, rigidity
is important. 
The category $\umod (A)$ is known to be rigid. 
In this section we show that $\ulrp (\Lambda^{\env})$ is rigid (cf.\ \cite[Exercise 2.10.16]{EGNO}) 
and describe some functors between $\umod (A)$ and $\ulrp (A^{\env})$.
These foundational results will help us, in subsequent sections, to 
study the monoidal triangular geometry of $\ulrp (A^{\env})$ and $\Thick(A) \subseteq \ulrp (A^{\env})$. 

\subsection{Left and right duals}
\label{subsect:duals}
First we recall definitions of left and right dual objects in a monoidal
(triangulated) category $\K$, where for the sake of simplicity, we regard
$\mathfrak{l}$ and $\mathfrak{r}$ as identity maps.
\begin{defn}
\begin{enumerate}[\rm(i)]
\item  An object $x$ in $\K$ has a \emph{left dual} object $x\ldual$ in $\K$ 
    if there are morphisms
    \[\ev_x\colon x\ldual\otimes x\to \mathfrak{e} \
    \text{ and } \ 
    \coev_x\colon \mathfrak{e}\to x\otimes x\ldual \]
    called \emph{evaluation} and \emph{coevaluation} morphisms, for which
    the compositions
    \[x\extto{\coev_x\otimes 1_x} (x\otimes x\ldual )\otimes
      x\extto{\mathfrak{a}} x\otimes (x\ldual\otimes x)\extto{1_x\otimes
        \ev_x} x \]
    and
 \[x\ldual\extto{1_{x\ldual}\otimes\coev_x} x\ldual\otimes (x\otimes
      x\ldual)\extto{\mathfrak{a}^{-1}} (x\ldual\otimes x)\otimes x\ldual\extto{\ev_x\otimes
        1_{x\ldual}} x\ldual \]
    are identity morphisms.
  \item An object $x$ in $\K$ has a \emph{right dual} object $\rdual x$ in $\K$ if there are
    morphisms
    \[\ev'_x\colon x\otimes {\rdual x}\to \mathfrak{e}\
    \text{ and }
    \coev'_x\colon \mathfrak{e}\to {\rdual x}\otimes x\]
    for which the compositions
    \[x\extto{1_x\otimes\coev'_x} x\otimes (\rdual x\otimes
      x)\extto{\mathfrak{a}^{-1}} (x\otimes {\rdual x})\otimes
      x\extto{\ev'_x\otimes 1_x} \rdual x\]
    and
 \[{\rdual x}\extto{\coev'_x\otimes 1_{\rdual x}} (\rdual x\otimes x)\otimes
      \rdual x\extto{\mathfrak{a}} \rdual x\otimes (x\otimes \rdual x)\extto{1_{\rdual x}\otimes\ev'_x} x\]
    are identity morphisms.
  \item The category $\mathcal{K}$ is \emph{rigid} if each
    object in $\mathcal{K}$ has a left and a right dual. 
    \end{enumerate}
  \end{defn}

\subsection{Dualities for finite dimensional selfinjective algebras}
\label{subsect:dual-selfinj}

Let $\Lambda$ be a finite dimensional selfinjective algebra over a field $k$. 
The enveloping algebra $\Lambda^\env = \Lambda\otimes_k \Lambda^{\opp}$ is also selfinjective.
We recall the notation as in Section \ref{sect:ex}: 
$\smod(\Lambda^{\env})$ is the category of finitely generated $\Lambda^{\env}$-modules (equivalently $\Lambda$-bimodules) and 
$\lrp(\Lambda^{\env})$ is  
the left-right projective category, that is the full subcategory of
$\smod(\Lambda^{\env})$ consisting of all $\Lambda$-bimodules $B$ such that
$_\Lambda B$ and $B_\Lambda$ are projective. 
We will work with the stable left-right projective category $\K = \ulrp(\Lambda^{\env})$
and $\A = \umod(\Lambda)$. Then $\K$ is a \mtc\ 
with tensor product $-\otimes_\Lambda -$ and shift functor
$\Sigma = \Omega^{-1}_{\Lambda^\env}(-)$, and $\A$ is a triangulated
category with shift functor $\Omega_\Lambda^{-1}(-)$. The category
$\K$ acts on $\A$ with action
$-*-$ given by $-\otimes_\Lambda -\colon \K\times\A \to \A$.

In this section we define duals in $\ulrp(\Lambda^{\env})$. While existence of duals for bimodules is known \cite[Exercise 2.10.16]{EGNO}, we include the details for the sake of completeness.
The next result is the first step in this direction. 
Let 
\[{\rdual (-)} = \Hom_\Lambda( {_\Lambda(-)}, {_\Lambda\Lambda})\colon \lrp(\Lambda^{\env}) \to \smod(\Lambda^{\env})\]
and
\[(-) \ldual = \Hom_\Lambda( {(-)_\Lambda}, {\Lambda_\Lambda})\colon \lrp(\Lambda^{\env}) \to \smod(\Lambda^{\env})\]
where the bimodule actions are given as follows. 
Let $B$ be in $\lrp(\Lambda^{\env})$.
We will show that $B\ldual = \Hom_\Lambda(B_\Lambda,
\Lambda_\Lambda)$ and $\rdual B = \Hom_\Lambda({_\Lambda
  B}, {_\Lambda \Lambda})$ are left and right dual objects of $B$,
respectively. Note that $B\ldual$ is a left
$\Lambda$-module via the action on the left of $\Lambda$, and it is a
right $\Lambda$-module via the action on the left of $B$. Hence
\[(\lambda\cdot f)(b) = \lambda f(b)\
\text{ and } \ 
(f\cdot \lambda)(b) = f(\lambda b)\]
for $f$ in $B\ldual$ and $\lambda$ in $\Lambda$. Note that $\rdual B$ is a left $\Lambda$-module via the
action on the right of $B$, and it is a right $\Lambda$-module via the
action on the right of $\Lambda$. Hence
\[(\lambda\cdot g)(b) = g(b\lambda) \ 
\text{ and } \
(g\cdot \lambda)(b) = g(b)\lambda\]
for $g$ in $\rdual B$ and $\lambda$ in $\Lambda$.

\begin{prop}\label{prop:dualityonlrp}
Let $\Lambda$ be a finite dimensional selfinjective algebra over a field $k$. 
There are isomorphisms of $\Lambda$-bimodules for all $P$ and $P'$ in $\lrp(\Lambda^{\env})$,
\[P\cong (\rdual P)\ldual\quad \textrm{and} \quad   P' \cong {\rdual(P' \ldual)} .\]

\end{prop}

\begin{proof}
When applied to a finitely generated projective left $\Lambda$-module $P$, it is classical
that
$P \cong {(\rdual P)}\ldual$
as left $\Lambda$-modules, and similarly for a finitely generated projective right
$\Lambda$-module $P'$, there is an isomorphism 
$P'\cong {\rdual(P'\ldual)}$
of right $\Lambda$-modules.  There are natural maps
\[\alpha_P\colon P\to {(\rdual P)}\ldual  = \Hom_\Lambda(\Hom_\Lambda({_\Lambda
    P},{_\Lambda \Lambda}),\Lambda_\Lambda)\]
and
\[\beta_{P'}\colon P'\to {\rdual(P'\ldual)} =
  \Hom_\Lambda(\Hom_\Lambda({P'_\Lambda},\Lambda_\Lambda),{_\Lambda\Lambda}),\]
given by
\[\alpha_P(p)(f) = f(p)\]
for $p\in P$ and $f\in \Hom_\Lambda({_\Lambda P},{_\Lambda\Lambda})={\rdual P}$
and
\[\beta_{P'}(p')(f') = f'(p')\]
for $p'\in P'$ and $f'\in
\Hom_\Lambda({P'_\Lambda},{\Lambda_\Lambda})=P' \ldual$. These homomorphisms are
isomorphisms, respectively as left and right $\Lambda$-modules. To
complete the proof it remains to show that they are
homomorphisms of bimodules. For $p$ in $P$ and $f$ in $\rdual P$, 
\[
(\alpha_P(p)\lambda)(f)  = \alpha_P(\lambda\cdot f)(p)\notag
  = f(p\lambda)\notag
   =  \alpha_P(p\lambda)(f)\notag  .
\]
This shows that $\alpha_P$ is a homomorphism of right
$\Lambda$-modules, and therefore a homomorphism of
$\Lambda$-bimodules.

Similarly, for $p$ in $P$ and $g$ in $P\ldual$,
\[
(\lambda\beta_P(p))(g)  = \beta_P(g\cdot\lambda)(p)\notag
   = g(\lambda p)\notag
   =  \beta_P(\lambda p)(g)\notag  .
\]
This shows that $\beta_P$ is a homomorphism of left
$\Lambda$-modules, and therefore a homomorphism of
$\Lambda$-bimodules. This completes the proof of the proposition. 
\end{proof}

The next result is essentially \cite[Proposition II.4.4]{ARS} 
extended to bimodules. It is used in proving that the functors
$\rdual (-)$ and $(-) \ldual $ are endofunctors of $\lrp(\Lambda^{\env})$.

\begin{prop}\label{prop:tensorhomrelation}
Let $\Lambda$ be a finite dimensional selfinjective algebra over a
field $k$.
\begin{enumerate}[\qquad \rm(a)]
\item For any $B$ in $\lrp(\Lambda^{\env})$ and $C$ in $\smod(\Lambda)$, the
  homomorphism
\[\alpha'_{B,C}\colon {\rdual B} \otimes_\Lambda C \to \Hom_\Lambda({_\Lambda B},{_\Lambda
    C}),\]
given by
$ \alpha'_{B,C}(f\otimes c)(b) = c\cdot f(b)$ 
for all $f$ in $\rdual B$, $c$ in $C$ and $b$ in $B$, is an
isomorphism in $\lrp(\Lambda^{\env})$ which is natural in both variables.
\item For any $B$ and $C$ in $\lrp(\Lambda^{\env})$ the homomorphism
\[\alpha_{B,C}\colon C\otimes_\Lambda B\ldual  \to
  \Hom_\Lambda(B_\Lambda, C_\Lambda),\]
given by
$ \alpha_{B,C}(c\otimes f)(b) = f(b)\cdot c$ 
for all $c$ in $C$, $f$ in $B \ldual$ and $b$ in $B$, is an isomorphism in
$\lrp(\Lambda^{\env})$ which is natural in both variables. 
\end{enumerate}
\end{prop}

Now we are ready to prove the main result in this section,
that $\rdual (-)$ and $(-) \ldual$ are endofunctors on $\lrp(\Lambda^{\env})$
and induce endofunctors on the stable category $\ulrp(\Lambda^{\env})$.

\begin{prop}\label{prop:endo-functors}
Let $\Lambda$ be a finite dimensional selfinjective algebra over a field $k$. 
\begin{enumerate}[\qquad \rm(a)]
\item If $B$ is in $\lrp(\Lambda^{\env})$, then $\rdual B$ and $B\ldual$ are in
  $\lrp(\Lambda^{\env})$. 
\item If $B$ is a projective object in $\lrp(\Lambda^{\env})$, then $\rdual P$ and
$P\ldual$ are both projective as $\Lambda^\env$-modules.
\item The functors ${\rdual (-)}$ and $(-)\ldual $ induce triangulated endofunctors on
  $\ulrp(\Lambda^{\env})$.
\item The functors $\rdual (-)$ and $(-)\ldual $ induce endofunctors on
  $\mathcal{E}=\Thick(\Lambda) \subseteq \lrp(\Lambda^{\env})$.
\item The functors $\rdual (-)$ and $(-)\ldual $ induce triangulated endofunctors on
  $\underline{\mathcal{E}} \subseteq \ulrp(\Lambda^{\env})$.
\end{enumerate}
\end{prop}

\begin{proof}
  (a) By Proposition \ref{prop:tensorhomrelation} (b),
  $-\otimes_\Lambda B\ldual \cong \Hom_\Lambda(B_\Lambda,-).$
The latter functor is exact, since $B_\Lambda$ is a projective right 
$\Lambda$-module. This proves that $_\Lambda{(\rdual B)}$ is a projective
left $\Lambda$-module.

By $\Hom$-$\otimes$-adjunction, 
\[\Hom_\Lambda(-,(B\ldual) _{\Lambda}) \cong \Hom_\Lambda(-\otimes_\Lambda
  B,\Lambda_\Lambda) \cong \Hom_\Lambda(-,\Lambda_\Lambda)\comp
  (-\otimes_\Lambda B),\]
where the composition consists of two exact functors since $\Lambda$ is
selfinjective and $B$ is in $\lrp(\Lambda^{\env})$. Hence $(B\ldual) _{\Lambda}$ is an
injective $\Lambda$-module and therefore a projective
$\Lambda$-module. Similarly one can prove that $\rdual B$ is in
$\lrp(\Lambda^{\env})$. 

(b) Note the following isomorphisms
of $\Lambda$-bimodules for idempotents $e$ and $f$ in $\Lambda$: 
\begin{align}
  \Hom_\Lambda({_\Lambda\Lambda e\otimes_k f\Lambda},
  {_\Lambda\Lambda}) & \cong \Hom_k(f\Lambda, \Hom_\Lambda(
                       {_\Lambda \Lambda e}, {_\Lambda
                       \Lambda}))\notag\\
                     & \cong \Hom_k(f\Lambda, e\Lambda)\notag\\
                     & \cong D(f\Lambda)\otimes_k e\Lambda\notag
\end{align}
where $D$ denotes the vector space dual.
Since $\Lambda$ is selfinjective, the last bimodule is a projective
$\Lambda^\env$-module. Similarly we obtain an isomorphism of
$\Lambda$-bimodules for idempotents $e$ and $f$ in $\Lambda$:
\[\Hom_\Lambda( {_\Lambda\Lambda e}\otimes_k f\Lambda_\Lambda,
  \Lambda_\Lambda) \cong \Lambda f\otimes_k D(\Lambda e).\]
Again, since $\Lambda$ is selfinjective, the last module is a
projective $\Lambda^\env$-module, and this completes the proof of part
(b).

(c) This follows directly from (b).

(d) Let $B$ be in $\mathcal{E}= \Thick(\Lambda)$. We want to prove that $\rdual B$ and $B\ldual $
again are in $\mathcal{E}$. Note that 
\[\rdual \Lambda \cong \Lambda \cong \Lambda \ldual\]
and
\[\rdual (\Omega^i_{\Lambda^\env}(\Lambda)) \cong
  \Omega^{-i}_{\Lambda^\env}(\Lambda) \cong
  (\Omega^i_{\Lambda^\env}(\Lambda))\ldual \]
for all $i\neq 0$. 
Since all objects in $\mathcal{E}$ are built from the shifts
$\Omega^i_{\Lambda^\env}(\Lambda)$ $(i\in\mathbb{Z})$ through direct
sums, direct summands, and cones, it follows that $\rdual B$ and $B\ldual $
again are in $\mathcal{E}$. 

(e) This follows directly from (c) and (d). 
\end{proof}

\subsection{Left and right duals in $\ulrp (\Lambda^{\env})$}  

This section is devoted to showing that $\ulrp (\Lambda^{\env})$ and the
subcategory $\mathcal{E} = \Thick(\Lambda)$ are rigid triangulated categories. 
Define
\[ 
    \ev_B\colon B\ldual\otimes_\Lambda B\to \Lambda \ \text{ and } \ 
    \ev'_B\colon B\otimes_\Lambda {\rdual B}\to \Lambda
\]
by $\ev_B(f\otimes b) = f(b)$ for all $f\in
B\ldual =\Hom_\Lambda(B,\Lambda)$ and $b\in B$ and 
$\ev_B(b\otimes g) = g(b)$ for all $b\in B$ and $g\in
{\rdual B}=\Hom_\Lambda(B,\Lambda)$. 
Note that 
\[\ev_B(\lambda (f\otimes b)) = \ev_B((\lambda f)\otimes b) = (\lambda
  f)(b) = \lambda f(b) = \lambda \ev_B(f\otimes b)\]
and
\[\ev_B((f\otimes b)\lambda) = \ev_B(f\otimes b\lambda) = f(b\lambda)
  = f(b)\lambda = \ev_B(f\otimes b)\lambda\]
for all $\lambda$ in $\Lambda$, $f$ in $B\ldual $ and $b$ in $B$, and that
\[\ev'_B(\lambda (b\otimes g)  = \ev'_B(\lambda b\otimes g) =
  g(\lambda b) = \lambda g(b) = \lambda \ev'_B(b\otimes g)\]
and
\[\ev'_B((b\otimes g)\lambda) = \ev'_b(b\otimes g\cdot \lambda) =
  g(b)\lambda = \ev'_B(b\otimes g) \lambda\]
for all $b$ in $B$, $g$ in $\rdual B$ and $\lambda$ in $\Lambda$. This shows that $\ev_B$
and $\ev'_B$ are homomorphisms of $\Lambda$-bimodules. 

Recall the isomorphism from Proposition~\ref{prop:tensorhomrelation}(b),
\[\alpha_{B,C}\colon C\otimes_\Lambda B\ldual \to \Hom_\Lambda(
  B_\Lambda, C_\Lambda),\]
for $B$ and $C$ in $\lrp(\Lambda^{\env})$.
Now we take $C=B$ and $b_i\otimes f_i$ in
$B\otimes_\Lambda B\ldual $ for $i=1,2,\ldots,v$ such that
\[\sum_{i=1}^v f_i(-)b_i\colon B\to B\]
is the identity homomorphism $1_B$ in $\Hom_\Lambda(
B_\Lambda, B_\Lambda)$. A homomorphism
\[\Lambda\to \Hom_\Lambda(B_\Lambda, B_\Lambda)\cong
  B\otimes_\Lambda B\ldual \]
  is uniquely determined by its value on $1$ in $\Lambda$. 
Define a $\Lambda$-bimodule homomorphism $\Lambda\to \Hom_\Lambda(B_\Lambda, B_\Lambda)$
by taking $\lambda$ in $\Lambda$ to the right $\Lambda$-module homomorphism
  \[ \lambda\cdot - \colon B_\Lambda\to B_\Lambda . \]
Under the isomorphism of 
  $B\otimes_\Lambda B\ldual $ with $\Hom_\Lambda(
  B_\Lambda, B_\Lambda)$, the element $\sum_{i=1}^v b_i\otimes f_i$
  corresponds to the identity map on $B$. Hence the map  
  $ \coev_B\colon \Lambda \to B\otimes_\Lambda B\ldual $
  given by
  \[\coev_B(1) = \sum_{i=1}^v b_i\otimes f_i\]
  is a homomorphism of $\Lambda$-bimodules. 

Similar to the above, we consider the isomorphism
\[\alpha'_{B,C}\colon {\rdual B}\otimes_\Lambda C \to \Hom_\Lambda(
  {_\Lambda B}, {_\Lambda C})\]
for $B$ and $C$ in $\lrp(\Lambda^{\env})$ given by $f\otimes c \mapsto
\alpha'_{B,C}(f\otimes c)(-) = c\cdot f(-)$. 
Let $C=B$ and $f'_i\otimes b'_i$ in
$\rdual B\otimes_\Lambda B$ for $i=1,2,\ldots,u$ be such that
\[\sum_{i=1}^u b'_if'_i(-)\colon B\to B\]
is the identity homomorphism $1_B$ in $\Hom_\Lambda(
{_\Lambda B}, {_\Lambda B})$. A homomorphism
\[\Lambda\to \Hom_\Lambda({_\Lambda B}, {_\Lambda B})\cong
  {\rdual B}\otimes_\Lambda B\]
  is uniquely determined by its value on $1$ in $\Lambda$. 
Define a $\Lambda$-bimodule homomorphism $\Lambda\to \Hom_\Lambda(
  {_\Lambda B}, {_\Lambda B})$ by taking $\lambda$ in $\Lambda$
to the left $\Lambda$-module homomorphism 
  \[ -\cdot \lambda \colon {_\Lambda B}\to {_\Lambda B} . \]
Under the isomorphism of
  $\rdual B\otimes_\Lambda B$ with $\Hom_\Lambda(
  {_\Lambda B}, {_\Lambda B})$, the element $\sum_{i=1}^u f'_i\otimes
  b'_i$ corresponds to the identity on $B$. Hence the map 
  $\coev'_B\colon \Lambda \to {\rdual B}\otimes_\Lambda B$
  determined by 
  \[\coev'_B(1) = \sum_{i=1}^u f'_i\otimes b'_i\]
  is a homomorphism of $\Lambda$-bimodules. Therefore both
  $\coev_B$ and $\coev'_B$ are homomorphisms of $\Lambda$-bimodules
  for all $B$ in $\lrp(\Lambda^{\env})$. 

We are now ready to prove the main result in this section.

\begin{prop}
  Let $\Lambda$ be a finite dimensional selfinjective algebra. Then
  the categories $\ulrp(\Lambda^{\env})$ and $\mathcal{E} = \Thick(\Lambda)
  \subseteq \ulrp(\Lambda^{\env})$ are rigid monoidal triangulated
  categories.
\end{prop}

\begin{proof}
  Let $B$ be in $\lrp(\Lambda^{\env})$. From the above discussion, for $b$ in
  $B$, 
\begin{align}
b \extto{\coev_B\otimes 1_B} & \coev_B(1) \otimes b = \left (\sum_{i=1}^v
                               b_i\otimes f_i \right) \otimes b\notag\\
  \extto{\mathfrak{a}} & \sum_{i=1}^v b_i\otimes (f_i\otimes
                         b)\notag\\
  \extto{1_B\otimes\ev_B} & \sum_{i=1}^v b_if_i(b) = b\notag  ,
\end{align}
and for $g$ in $B\ldual$,
\begin{align}
g \extto{1_{B\ldual }\otimes \coev_B} & g\otimes \coev_B(1) = g\otimes
                                   \left (\sum_{i=1}^vb_i\otimes f_i \right)\notag\\
  \extto{\mathfrak{a}^{-1}} & \sum_{i=1}^v (g\otimes b_i)\otimes
                              f_i\notag\\
  \extto{\ev_B\otimes 1_{B\ldual }} & \sum_{i=1}^v g(b_i)\cdot f_i\notag  ,
\end{align}
where for $b$ in $B$,
\[\left (\sum_{i=1}^v g(b_i)\cdot f_i\right )(b) = \sum_{i=1}^v g(b_i)\cdot f_i(b)
  = \sum_{i=1}^v g(b_if_i(b)) = g\left (\sum_{i=1} b_if_i(b)\right) = g(b).\]
Hence both the compositions of morphisms are the identity morphisms
and $B\ldual $ is a left dual of $B$. 

For $b$ in $B$,
\begin{align}
  b  \extto{1_B\otimes \ev'_B} & b\otimes \coev'_B(1) = b\otimes
      \sum_{i=1}^u f'_i\otimes b'_i\notag\\
   \extto{\mathfrak{a}^{-1}} &\sum_{i=1}^u (b\otimes f'_i)\otimes
    b'_i\notag\\
     \extto{\coev'_B\otimes 1_B} &\sum_{i=1}^u f'_i(b)b_i = b\notag  ,
\end{align}
and for $g$ in $\rdual B$,
\begin{align}
  g \extto{\coev'_B\otimes 1_{\rdual B}} & \sum_{i=1}^u (f'_i\otimes
                                      b'_i)\otimes g\notag\\
  \extto{\mathfrak{a}} & \sum_{i=1}^u f'_i\otimes (b'_i\otimes
                         g)\notag\\
  \extto{1_{\rdual B}\otimes \ev'_B} & \sum_{i=1}^u f'_i\cdot g(b'_i)\notag  ,
\end{align}  
where for $b$ in $B$, 
\[\left (\sum_{i=1}^u f'_i\cdot g(b'_i)\right)(b) = \sum_{i=1}^u f'_i(b)\cdot
  g(b'_i) = \sum_{i=1}^u g(f'_i(b)b'_i) = g\left (\sum_{i=1}^u f'_i(b)b'_i\right)
  = g(b).\]
Hence both the compositions of morphisms are the identity morphisms
and $\rdual B$ is a right dual of $B$. The result then follows from the
above discussion and Proposition \ref{prop:endo-functors}. 
\end{proof}

\subsection{A functor from $\lrp(\Lambda^{\env})$ to $\smod(\Lambda)$}
\label{subsect:functorG}
Here we return to the problem of finding correspondences among the
four objects introduced in Section \ref{subsec:correspondences}. 
To that end, we define here a functor from $\ulrp(\Lambda^{\env})$ to $\umod(\Lambda)$ which is
conservative (that is, it takes nonzero objects to nonzero objects). 

First we recall some general $\Hom$-$\otimes$-isomorphisms (see \cite[Theorem 2.75]{Rotman}).
We provide a proof for completeness. 
 
\begin{thm}\label{thm:homtensor2theorem} 
Let $A = {_\Lambda A}_\Gamma$, $B = {_\Gamma B}_\Delta $ and $C = {_\Lambda
  C}_\Delta$ be bimodules for three $k$-algebras $\Lambda$, $\Gamma$ and
$\Delta$ over a field $k$.
\begin{enumerate}[\qquad \rm(a)]
\item The usual $\Hom$-$\otimes$-adjunction homomorphism
  \[\tau_{A,B,C}\colon \Hom_{\Lambda\otimes_k \Delta^{\opp}}(A\otimes_\Gamma B, C) \cong
    \Hom_{\Lambda\otimes_k \Gamma^{\opp}}(A, \Hom_{\Delta^{\opp}}(B,C))\]
given by $\tau_{A,B,C}(f)(a)(b) = f(a\otimes b)$ for $a\in A$ and $b\in B$ 
  is a functorial isomorphism in all variables.
\item If $A_\Gamma$ is a projective $\Gamma$-module and $B_\Delta$ is
  a projective $\Delta$-module, then the
  $\Hom$-$\otimes$-adjunction homomorphism induces an isomorphism
 \[\Ext^i_{\Lambda\otimes_k \Delta^{\opp}}(A\otimes_\Gamma B, C) \cong
    \Ext^i_{\Lambda\otimes_k \Gamma^{\opp}}(A, \Hom_{\Delta^{\opp}}(B,C))\]
for all $i\geq 0$ that is functorial in all variables.
\end{enumerate}
\end{thm}

\begin{proof}
(a) The proof is the same as for \cite[Theorem 2.75]{Rotman}, just
observing that all maps involved have the right structural
properties. 

(b) Assume that $A_\Gamma$ is a projective $\Gamma$-module. Let
$\mathbb{P}(A)$ be a projective resolution of $A$ as a
$\Lambda\otimes_k\Gamma^{\opp}$-module. Since $A_\Gamma$ is a projective
$\Gamma$-module, the projective resolution $\mathbb{P}(A)$ is a split
exact long exact sequence as right $\Gamma$-modules. Hence tensoring
with $B$, it remains exact. 
Since $B_\Delta$ is a projective $\Delta$-module, 
a projective 
$\Lambda\otimes_k\Gamma^{\opp}$-module $P$ tensored with $B$ is a projective
$\Lambda\otimes \Delta^{\opp}$-module, and 
we infer that $\mathbb{P}(A)\otimes_\Gamma B$ is a 
$\Lambda\otimes \Delta^{\opp}$-projective resolution of $A\otimes_\Gamma B$. Then (a)
implies an isomorphism of complexes
\[\Hom_{\Lambda\otimes_k\Delta^{\opp}}(\mathbb{P}(A)\otimes_\Gamma B, C)
  \cong
  \Hom_{\Lambda\otimes_k \Gamma^{\opp}}(\mathbb{P}(A), \Hom_{\Delta^{\opp}}(B,C)) . \]
The claim follows directly from this. 
\end{proof}

Let $\rad(\Lambda)$ denote the Jacobson radical of $\Lambda$, let $\Mod(\Lambda^{\env})$ denote the category of all (including infinite dimensional) modules for $\Lambda^{\env}$, and 
$\lrP(\Lambda^{\env})$ denote the full subcategory of $\Mod(\Lambda^\env)$
consisting of all bimodules projective as one-sided $\Lambda$-modules
on each side. 
Tensoring with $\Lambda/\rad(\Lambda)$ provides a functor 
\[-\otimes_\Lambda \Lambda/\rad(\Lambda)\colon
  \lrP(\Lambda^{\env})\to\Mod (\Lambda^\env) \]
that we study next. 

\begin{prop}
\label{prop:findimfunctor}
Let $\Lambda$ be a finite dimensional selfinjective algebra over a
field $k$, such that $\Lambda/\rad(\Lambda)$ is a separable algebra
over $k$. Consider the functor  
\[-\otimes_\Lambda \Lambda/\rad(\Lambda) \colon \lrP(\Lambda^{\env}) \to
  \Mod(\Lambda).\] 
Let $B$ be a module in $\lrP(\Lambda^{\env})$.
\begin{enumerate}[\qquad \rm(a)]
\item If $B\otimes_\Lambda \Lambda/\rad(\Lambda)$ is projective, then
  $B$ is a projective $\Lambda^\env$-module.
\item If $B$ is nonzero in $\ulrP(\Lambda^{\env})$, then
  $B\otimes_\Lambda\Lambda/\rad(\Lambda)$ is nonzero in
  $\uMod(\Lambda)$; equivalently, the functor
  \[-\otimes_\Lambda \Lambda/\rad(\Lambda) \colon \ulrP(\Lambda^{\env})\to
    \uMod(\Lambda)\]
  is conservative.
\end{enumerate}
\end{prop}

\begin{proof}
(a) Using Theorem \ref{thm:homtensor2theorem} when $B_\Lambda$ is a
projective $\Lambda$-module and $\Delta = k$, we have the
following isomorphism: 
\[\Ext^i_\Lambda(B\otimes_\Lambda
  \Lambda/\rad(\Lambda),\Lambda/\rad(\Lambda)) \cong
  \Ext^i_{\Lambda^\env}(B,\Hom_k(\Lambda/\rad(\Lambda),\Lambda/\rad(\Lambda)))\]
  for all $i>0$.
Note that 
\[\Hom_k(\Lambda/\rad(\Lambda),\Lambda/\rad(\Lambda)) \cong
  \Lambda/\rad(\Lambda)\otimes_k D(\Lambda/\rad(\Lambda)),\]
and any simple right $\Lambda$-module is a direct summand of $D(\Lambda/\rad(\Lambda))$. By \cite[Proposition 7.7]{CR}, using that $\Lambda/\rad(\Lambda)$ is separable over $k$, the
module
\[\Lambda/\rad(\Lambda)\otimes_k D(\Lambda/\rad(\Lambda))\]
is semisimple and any simple $\Lambda^\env$-module is a direct summand. It follows that $B$ is a projective $\Lambda^\env$-module. This proves the claim.

(b) This is a direct consequence of (a).
\end{proof}

\subsection{Modules and bimodules for a finite dimensional Hopf algebra $A$}
\label{subsect:functorF}
Here we take $A$ to be a finite dimensional Hopf algebra over a field $k$.
We are interested in categories of $A$-modules and $A$-bimodules. 
We will define some needed functors relating them. 

Let
\[
  F\colon \smod (A) \rightarrow \smod (A^{\env})
\]
denote the following tensor induction functor: for each $A$-module $M$,
\[
F(M) = A^{\env}\otimes_{\delta(A)} M ,
\]
where $\delta \colon A \rightarrow A^{\env}$ is the algebra homomorphism
given by
$
\delta (a) = \sum a_1\otimes S(a_2)
$
for all $a\in A$, embedding $A$ as a subalgebra of $A^{\env}$.

The following fact is known, but we provide a proof for completeness. 

\begin{prop}
  The image of $F$ is contained in $\lrp (A^{\env})$.
\end{prop}

\begin{proof}
Let $M$ be an $A$-module. 
First note that any element $(a\otimes b)\otimes_{\delta(A)} m$ of $F(M)$,
for $a,b\in A$ and $m\in M$, 
may be written in two other ways:
\begin{eqnarray*}
  (a\otimes b)\otimes_{\delta(A)} m & = &   \sum (a_1\otimes S(a_2)a_3b)
                \otimes_{\delta(A)}m\\
    &=&  \sum(1\otimes a_2 b)\otimes_{\delta(A)} a_1 m, \\
    (a\otimes b)\otimes_{\delta(A)} m & = & \sum
       (ab_3 S^{-1}(b_2)\otimes b_1)\otimes
    _{\delta(A)} m \\
      &=& \sum (ab_2\otimes 1)\otimes_{\delta(A)} S^{-1}(b_1)m .
      \end{eqnarray*}

Define an isomorphism of right $A$-modules
\begin{equation}\label{eqn:tr-iso}
\phi \colon F(M)\rightarrow M_{\tr}\otimes A ,
\end{equation}
where $M_{\tr}$ denotes
the vector space $M$ with trivial right 
$A$-module structure, via the counit $\varepsilon$.
Let 
$
\phi ( (1\otimes b) \otimes_{\delta(A)} \otimes m) = m\otimes b 
$
for all $b\in A$ and $m\in M$, and extend linearly.
Then by its definition, $\phi$ is a right $A$-module homomorphism,
and it can be seen to be an isomorphism. 
Since $M_{\tr}\otimes A$ is free as a right $A$-module, so is $F(M)$.

Next, we similarly define an isomorphism of left $A$-modules
$
\psi\colon F(M)\rightarrow A\otimes M_{\tr} ,
$
where this time $M_{\tr}$ denotes the vector space $M$
with trivial left $A$-module structure, by
$
\psi((a\otimes 1)\otimes_{\delta(A)} m ) = a\otimes m
$
for all $a\in A$ and $m\in M$. 
Since $A\otimes M_{\tr}$ is free as a left $A$-module, so is $F(M)$. 
\end{proof}

\begin{rem}\label{rem:tr-right}
  Under the isomorphism $\phi$ of right $A$-modules in the above proof,
  the induced $A$-bimodule structure on $M_{\tr}\ot A$ is given by 
  \[
  a\cdot (m\ot b) = \sum a_1m \ot a_2b \ \ \
  \mbox{ and } \ \ \ (m\ot b)\cdot a= m\ot ba .
  \]
  To see this, note that $\phi$ takes $(1\ot b)\ot_{\delta(A)}m$
  to $m\ot b$.
  Using the expression from the beginning of the proof above,
  the left $A$-action on $F(M)$ may be given as 
  \[
  a\cdot ( (1\ot b)\ot_{\delta(A)} m ) = (a\ot b)\ot_{\delta(A)}m
  = \sum (1\ot a_2b)\ot_{\delta(A)} a_1m ,
  \]
  and $\phi$ takes this element to $\sum a_1m \ot a_2b$, as desired.
  Similarly, for the right action, note that on $M_{\tr}\ot A$,
  \[
  (m\ot b)\cdot a = \sum m \varepsilon(a_1)\ot ba_2
  = \sum m \ot b\varepsilon(a_1)a_2 = m\ot ba .
  \]
\end{rem}

Now let
\[
G\colon  \smod (A^{\env}) \rightarrow \smod (A) 
\]
denote the functor given by action on the identity object 
$\unit = k$ of $\smod (A)$, that is, 
\[
G(Y) = Y\ot_A \unit  = Y * \unit  ,
\]
with left $A$-module action given by the left action on
the $A$-bimodule $Y$.

\begin{prop}\label{prop:Hopf-GF}
The functor $F$ is an exact monoidal functor that is right inverse to $G$,
that is $GF\cong 1_{\smod (A)}$.
\end{prop}

\begin{proof}
The functor $F$ is exact monoidal by~\cite[Section 4]{KaradagW2022};
in particular, $F(\unit_{\smod (A)} ) \cong \unit_{\smod (A^{\env})}$ and
$F(M\ot_k N) \cong F(M)\ot_{A} F(N)$ for all objects $M,N$ of $\smod (A)$.
Now let $M$ be any left $A$-module.
  By~(\ref{eqn:tr-iso}) and Remark~\ref{rem:tr-right},
  \[
    GF(M)  \cong    G (M_{\tr}\ot A) 
    =  (M_{\tr}\ot A)\ot_A   k 
    \cong   M_{\tr}\ot k .
  \]
  Again by Remark~\ref{rem:tr-right},
  the left $A$-module structure thus induced on
  this vector space $M_{\tr}\ot k\cong M$
  is indeed the original $A$-module structure on $M$,
  since the action on $k$ is via the counit $\varepsilon$.
\end{proof}

\begin{rem}
  Under a unipotent condition that will be imposed in a later section,
  that is the condition that 
$A$ has unique simple module $k$ up to isomorphism, 
the image of $F$ is in fact 
contained in $\E=\Thick(A)$ since $F$ takes extensions to extensions. 
  \end{rem}

\section{Meditations on unipotent monoidal triangular geometry}
\label{sec:unipotent}

In this section, we prove several general results for monoidal triangulated categories that are thickly 
generated by their tensor units. We will apply these results in Section~\ref{sec:unipotent-Hopf} to study the left-right projective category
$\ulrp (A^{\env})$ and $\Thick(\unit) \subseteq \ulrp(A^{\env})$, for $A$ a finite dimensional unipotent Hopf algebra.

\subsection{The shell}
Let $\K$ be a \mtc.
The thick subcategory $\mathcal{E} = \mathcal{E}(\K) = \Thick(\mathfrak{e})$ of $\K$
plays a central role in the theory developed later.
We study some of its properties here. 

First, we review the category $\mathcal{E}$ in the setting of group
algebras. Consider a group algebra $kG$ of a finite group $G$ and field
$k$ with characteristic dividing the order of $G$. 
Let $\widehat{\Ext}_{kG}$ denote the Tate extension groups. 
The full subcategory of $\umod(kG)$  generated by the indecomposable $kG$-modules $M$ in the
principal block $kG_0$ of $kG$ with
\[\widehat{\Ext}^{\bu}_{kG}(k,M) \neq 0\]
coincides with the thick subcategory of $\umod (kG)$ generated by
$k$ (see \cite[Comment before Proposition 4.1]{BCR}, \cite[Sections 2
and 3]{BCRo}).  Any module $N$ in the principal block of $kG$ with 
\[\widehat{\Ext}^{\bu}_{kG}(k,N) = 0\]
is in the thick subcategory of modules with variety in the nucleus
\cite[Proposition 4.4]{BCR}. This motivates calling the thick
subcategory generated by the tensor identity $k$ the \emph{shell} of
$\umod(kG)$. Next we generalize this to any \mtc\ $\K$. 

For any object $x$ in $\mathcal{E} = \mathcal{E}(\K)$ there exists an integer $n$
such that
\[\Hom_\K(\Sigma^n(\mathfrak{e}),x) \neq 0\textup{\ or equivalently\ }
  \Hom_\K^{\mathbb{Z}}(\mathfrak{e},x)\neq 0.\]
By definition, since $\mathcal{E} = \Thick(\mathfrak{e})$, the object
$\mathfrak{e}$ is a classical generator of $\mathcal{E}$. Then
$\mathfrak{e}$ is also a generator
of $\mathcal{E}$ (in the sense of e.g.~ \cite[Section 3.1.1]{Rouquier}), that is, for any
nonzero object $x$ in $\mathcal{E}$ there exists an integer $n$ and a
nonzero morphism $\mathfrak{e}\to \Sigma^n(x)$. Consequently, there
is a nonzero morphism $\Sigma^{-n}(\mathfrak{e})\to x$. This
motivates calling $\mathcal{E}$ the \emph{shell} of
$\mathcal{K}$ and
\[
\{y\in \K\mid
  \Hom_{\mathcal{K}}^{\mathbb{Z}}(\mathfrak{e},a) = 0\}\]
the \emph{nucleus} of $\mathcal{K}$.

The thick subcategory $\mathcal{E}$ is in addition a monoidal
thick triangulated subcategory of $\K$, a well-known fact we record in
the next lemma. The second claim in the lemma follows using an argument analogous 
to the proof of Proposition \ref{prop:endo-functors}(d).

\begin{lem}\label{lem:montrisubcat}
Let $\K$ be a \mtc. Then $\mathcal{E} =
\mathcal{E}(\mathcal{K})$ is a monoidal thick triangulated subcategory
of $\K$. If $\K$ is rigid, then $\E$ is rigid, and the duality functors on $\E$ 
are the restrictions of the duality functors on $\K$.
\end{lem}

Recall that an (abelian) tensor category is called {\emph{unipotent}} if its only simple 
object is its tensor unit, see \cite[Definition 2.1]{EG}. We make the analogous definition for 
monoidal triangulated categories.

\begin{defn}
\label{defn:unipotent}
Let $\K$ be a \mtc. We call $\K$ {\emph{unipotent}} if $\K = \E(\K)$. 
\end{defn}

In particular, by Lemma \ref{lem:montrisubcat},
the shell $\E(\K)$ of a \mtc\ $\K$ is a unipotent \mtc.

For (a) and (b) in the following result, compare with \cite[Lemma
3.13]{Stevenson}.

\quad

\begin{lem}\label{cor:propofmathcalE}
  Let $\K$ be a \mtc\ acting on a triangulated
  category $\A$.  Let $\mathcal{E} = \mathcal{E}(\mathcal{K})$.
\begin{enumerate}[\qquad \rm(a)]
\item Any thick subcategory $\mathcal{U}$ of $\K$ has the property
  that
  \[\mathcal{E}\otimes \mathcal{U} = \mathcal{U} = \mathcal{U}\otimes
    \mathcal{E},\]
  that is, any thick subcategory in $\mathcal{E}$ is an ideal in
  $\mathcal{E}$.  
\item Any thick subcategory $\mathcal{U}$ of $\A$ has the property
  that
  \[\mathcal{E} * \mathcal{U} = \mathcal{U},\]
  that is, any thick subcategory in $\mathcal{A}$ is an
  $\mathcal{E}$-submodule of $\mathcal{A}$. 
\item If $\P$ is a prime ideal in $\mathcal{E}$, then $\P$ is
    completely prime. In other words, the category $\mathcal{E}$ is completely
    prime.  
\end{enumerate}
\end{lem}

We note some elementary consequences for $\E$ which we will use later.

\begin{lem} \label{lem:tensor-int}
Suppose that $\K$ is a rigid \mtc. 
If $\I$ and $\J$ are thick ideals of $\K$, then $ \I \otimes \J = \langle \I \cap \J\rangle_{\K}.$ 
If $\I$ and $\J$ are thick ideals  of $\E$, then $\I \otimes \J = \I \cap \J$. In particular, the tensor product of thick ideals commutes.
\end{lem}

\begin{proof}
    It is clear from the definition of thick ideal that $ \I \otimes \J \subseteq \I \cap \J$. Since $\K$ is rigid, by \cite[Proposition 4.1.1]{NVY2} every thick ideal is the intersection of all primes containing it. Therefore,
    \begin{align*}
        \I \cap \J &= \left ( \bigcap_{\I \subseteq \P} \P\right ) \cap \left ( \bigcap_{\J \subseteq \P} \P \right ) \\
        &= \bigcap_{\I \text{ or } \J \subseteq \P} \P\\
        &= \bigcap_{\I \otimes \J \subseteq \P} \P\\
        &= \bigcap_{\langle \I \otimes \J \rangle_{\K} \subseteq \P} \P\\
        &= \langle \I \otimes \J\rangle_{\K},
    \end{align*}
where each intersection is taken over all $\P \in \Spc (\K)$ satisfying the given condition. Since every thick subcategory of $\E$ is a thick ideal by Lemma~\ref{cor:propofmathcalE}(a), if $\I$ and $\J$ are thick ideals of $\E$ then $\langle \I \otimes \J \rangle_{\E} = \I \otimes \J$, and the second claim follows.
\end{proof}

\begin{lem}
\label{lem:tensor-intersection}
Suppose that $\E$ is rigid, and let $x$ and $y$ be objects of $\E$.  
Then $\langle x \otimes y \rangle_{\E} = \langle x \rangle_{\E} \otimes \langle y \rangle_{\E} = \langle x \rangle_{\E} \cap \langle y \rangle_{\E}= \Thick(x) \cap \Thick(y)$.
\end{lem}

\begin{proof}
    Since all prime ideals of $\E$ are completely prime by Lemma~\ref{cor:propofmathcalE}(c), 
    \[
    \langle x \rangle_{\E} \otimes \langle y\rangle_{\E} \subseteq \P \Leftrightarrow \langle x \rangle_{\E} \text{ or }\langle y \rangle_{\E} \subseteq \P \Leftrightarrow x \text{ or } y \in \P \Leftrightarrow x \otimes y \in \P
    \]
    for any prime ideal $\P$. Since all thick subcategories of $\E$ are thick ideals by Lemma~\ref{cor:propofmathcalE}(a), $\langle x \rangle_{\E} \otimes \langle y \rangle_{\E}$ is a thick ideal. Since all thick ideals are semiprime by \cite[Proposition 4.1.1]{NVY2}, $\langle x \rangle_{\E} \otimes \langle y \rangle_{\E}$ is the intersection of prime ideals containing it, and $\langle x \otimes y \rangle_{\E}$ is the intersection of primes containing it, i.e.~ primes containing $x \otimes y$. Hence $\langle x \otimes y \rangle_{\E}=\langle x \rangle_{\E} \otimes \langle y \rangle_{\E}$. The remaining equalities follow from Lemma~\ref{cor:propofmathcalE}(a) and Lemma \ref{lem:tensor-int}. 
\end{proof}

As an immediate consequence of Lemma \ref{lem:tensor-intersection}, as stated in the
following corollary, ideals in $\E$ generated by a tensor product do not depend on the order of the tensorands.

\begin{cor}
    \label{cor:either-side-equal}
Suppose $\E$ is rigid. For all $x$ and $y$ in $\E$, $\langle x \otimes y \rangle_{\E} = \langle y \otimes x \rangle_{\E}.$
\end{cor}

\subsection{Non-monoidal unipotent subcategories}

We now suppose that $(\K_1,\otimes_1,\unit_1)$ and $(\K_2, \otimes_2, \unit_2)$ are two rigid \mtc's such that $\K_1$ is a full triangulated subcategory of $\K_2$ with inclusion functor not necessarily monoidal, and that $\K_i = \E(\K_i)=\Thick(\unit_i)$ for $i=1,2$, that is, both $\K_1$ and $\K_2$ are unipotent. We give an abstract description of $\Spc (\K_1)$ in terms of $\Spc (\K_2)$ and prove a compatibility condition with the comparison maps from Balmer spectra to cohomological spectra. The motivation is to study, 
for $A$ a finite dimensional unipotent Hopf algebra,
the shell $\E$ of the stable left-right projective category $\ulrp(A^{\env})$ under the monoidal product $\ot_A$,
which is a non-monoidal triangulated subcategory of $\umod(A^{\env})$
under the monoidal product $\ot_k$.

We first show that thick ideals generated by tensor products do not depend on the choice of tensor product. Note that here and throughout this section, by Lemma \ref{cor:propofmathcalE}(a), if $x \in \K_1$ then $\langle x \rangle_{\K_1} = \Thick(x) = \langle x \rangle_{\K_2}$, so the notation $\langle x \rangle$ is unambiguous.

\begin{cor}
\label{cor:ideal-both-equal}
    For all $x$ and $y$ in $\K_1$, we have $\langle x \otimes_1 y \rangle = \langle x \otimes_2 y\rangle.$  \end{cor}

\begin{proof}
    This follows directly from Lemma \ref{lem:tensor-intersection}, since both sides of the equality are equal to $\Thick(x) \cap \Thick(y)$, which is independent of the tensor product used.
\end{proof}

Denote the Balmer support $\supp_{\K_i}$ on $\K_i$ by $\supp_i$ for $i=1,2$. We now realize $\Spc (\K_1)$ as a closed subspace of $\Spc (\K_2)$.

\begin{prop}
\label{prop:support-subcat-spc}
    There are homeomorphisms $\Spc (\K_1) \cong \supp_2(\unit_1)$ given by inverse maps 
    \begin{center}
        \begin{tikzcd}
\P \arrow[r, maps to, "\eta"]          & \{x \in \K_2 \mid x \otimes_2 \unit_1 \in \P\} \\
\Spc (\K_1) \arrow[r, bend left] & \supp_2(\unit_1) \arrow[l, bend left]             \\
\Q \cap \K_1                  & \Q \arrow[l, maps to, "\zeta"]                       
\end{tikzcd}
    \end{center}
\end{prop}

\begin{proof}
We check first that the claimed maps are well defined. 

Indeed, we must show that if $\Q \in \supp_2(\unit_1)$, then $\Q \cap \K_1$ is a prime for $\K_1$. Suppose $x$ and $y$ are in $\K_1$ with $x \otimes_1 y \in \Q \cap \K_1$. Then $x \otimes_2 y \in \Q \cap \K_1$ by Corollary \ref{cor:ideal-both-equal}. Therefore, $x$ or $y$ is in $\Q$ since $\Q$ is prime in $\K_2$. It follows that $x$ or $y$ is in $\Q \cap \K_1$. Note that $\Q \cap \K_1$ is properly contained in $\K_1$ since $\Q \in \supp_2(\unit_1)$, i.e.~ $\unit_1 \not \in \Q$ and is consequently not in $\Q \cap \K_1$. Therefore, $\Q \cap \K_1 \in \Spc (\K_1)$. 

On the other hand, we will show that if $\P \in \Spc (\K_1)$, then 
\[
\eta(\P):=\{x \in \K_2 \mid x \otimes_2 \unit_1 \in \P\}
\]
 is in $\supp_2(\unit_1)$, i.e.~ it is a prime ideal of $\K_2$ which does not contain $\unit_1$. Suppose $x$ and $y$ are in $\K_2$, and $x \otimes_2 y \in \eta(\P).$ By definition, $x \otimes_2 y \otimes_2 \unit_1 \in \P$. Of course, this implies $\unit_1 \otimes_2 x \otimes_2 y \otimes_2 \unit_1 \in \P$. Then $(\unit_1 \otimes_2 x) \otimes_1 (y \otimes_2 \unit_1) \in \P$ by Corollary \ref{cor:ideal-both-equal} and the fact that both $\unit_1 \otimes_2 x$ and $y \otimes_2 \unit_1$ are in $\K_1 = \Thick(\unit_1) = \langle \unit_1 \rangle_2$. Since $\P$ is a prime for $\K_1$, either $\unit_1 \otimes_2 x $ or $y \otimes_2 \unit_1$ is in $\P$. Note that if $\unit_1 \otimes_2 x \in \P$, then $x \otimes_2 \unit_1$ is in $\P$ by Corollary \ref{cor:either-side-equal}. Either way, we have $x \otimes_2 \unit_1$ or $y \otimes_2 \unit_1 \in \P$, i.e.~ $x$ or $y$ is in $\eta(\P)$; hence $\eta(\P)$ is prime. Note that $\unit_1$ cannot be in $\eta(\P)$, since if $\unit_1 \otimes_2 \unit_1 \in \P$ then $\unit_1 \otimes_1 \unit_1 \in \P$ (again using Corollary \ref{cor:ideal-both-equal}), in other words, $\unit_1 \in \P$, which would imply $\P = \K_1$. Therefore, $\eta(\P) \in \supp_2(\unit_1)$, and the maps given in the statement of the theorem are well defined. 

Note that $\zeta$ is continuous: if $\mS_1$ is a collection of objects in $\K_1$, then
\begin{align*}
    \zeta^{-1}(\supp_1(\mS_1)) &=\{ \Q \in \supp_2(\unit_1) \mid \Q \cap \K_1 \cap \mS_1 = \varnothing\}\\
    &=\supp_2(\unit_1) \cap \supp_2(\mS_1)
\end{align*}
which is closed in $\Spc (\K_2).$ On the other hand, $\eta$ is also continuous: if $\mS_2$ is a collection of objects in $\K_2$, then
\begin{align*}
\eta^{-1}(\supp_2(\mS_2)) &= \{ \P \in \Spc (\K_1) \mid \eta(\P) \cap \mS_2 = \varnothing\}\\
&= \{ \P \in \Spc (\K_1) \mid x \otimes_2 \unit_1 \not \in \P\;\; \forall \;\; x \in \mS_2\}\\
&= \supp_1 ( \{x \otimes_2 \unit_1 \mid x \in \mS_2\}),
\end{align*}
which is closed in $\Spc (\K_1)$. Hence $\eta$ and $\zeta$ are both continuous.

Lastly, we check that $\eta$ and $\zeta$ are inverse to one another:
for all $\P \in \Spc (\K_1)$,
\begin{align*}
    (\zeta \circ \eta)(\P) &=\{x \in \K_2 \mid x \otimes_2 \unit_1 \in \P\} \cap \K_1\\
    &= \{x \in \K_1 \mid x \otimes_2 \unit_1 \in \P\}\\
    &= \{x \in \K_1 \mid x \otimes_1 \unit_1 \in \P\}\\
    &=\{ x \in \K_1 \mid x \in \P\}\\
    &= \P ,
\end{align*}
where the third equality holds by Corollary \ref{cor:ideal-both-equal}. On the other hand,
\begin{align*}
   (\eta \circ \zeta)(\Q) &= \{x \in \K_2 \mid x \otimes_2 \unit_1 \in \Q \cap \K_1\} \\
   &= \{x \in \K_2 \mid x \otimes_2 \unit_1 \in \Q \}\\
   &= \{x \in \K_2 \mid x \in \Q \text{ or } \unit_1 \in \Q\}\\
   &=\{x \in \K_2\mid x \in \Q\}\\
   &= \Q
\end{align*}
for all $\Q \in \supp_2(\unit_1)$ (the third equality uses the fact that $\Q$ is prime in $\K_2$, and the fourth uses the fact that $\Q \in \supp_2(\unit_1)$, that is, $\unit_1 \not \in \Q$). This completes the proof. 
\end{proof}

Recall from Section \ref{subsect:catcenter} that there are continuous maps
\[
\rho_i\colon \Spc (\K_i) \to \Spech (\cohom(\K_i))
\]
where $\cohom(\K_i)$ is the cohomology ring of $\K_i$. (Since both $\K_1$, $\K_2$ are assumed to be generated by their respective tensor units, 
we may take $\ccent (\K_i) = \cohom (\K_i)$ for each.)
There is a homomorphism of graded rings
\begin{align*}
    \cohom(\K_2 ) &\to \cohom(\K_1)\\
    f & \mapsto f \otimes_2 1_{\unit_1},
\end{align*}
since for $f\colon \unit_2 \to \Sigma^i \unit_2$ a homogeneous element of $\cohom(\K_2)$ we can view $f \otimes_2 1_{\unit_1}$ as a map $\unit_1 \to \Sigma^i \unit_1$ using the isomorphisms
\[
\unit_2 \otimes_2 \unit_1 \cong \unit_1, \quad (\Sigma^i \unit_2) \otimes_2 \unit_1 \cong \Sigma^i( \unit_2 \otimes_2 \unit_1) \cong \Sigma^i \unit_1.
\] 
This induces a continuous map
\begin{align}
\label{map1}
   \iota \colon  \Spech(\cohom(\K_1)) & \to \Spech (\cohom(\K_2))\\
    \mfp & \mapsto \{f \in \cohom(\K_2) \mid f \otimes_2 1_{\unit_1} \in \mfp \}
\end{align}
for homogeneous prime ideals $\mfp$ of $\cohom(\K_1)$.

\begin{prop}
\label{prop:rho-diag}
    The diagram
\begin{center}
    \begin{tikzcd}
\Spc (\K_1)  \arrow[d, "\eta"] \arrow[r, "\rho_1"] & \Spech(\cohom(\K_1)) \arrow[d, "\iota"] \\
\Spc (\K_2) \arrow[r, "\rho_2"]                              & \Spech(\cohom(\K_2))          
\end{tikzcd}
\end{center}
commutes, where the left vertical map is the inclusion induced by the map $\eta$ from Proposition~\ref{prop:support-subcat-spc}, and the right vertical map is the map $\iota$ from (\ref{map1}).
\end{prop}

\begin{proof}
    Let $\P \in \Spc (\K_1)$. By traversing the diagram first right and then down, we compute
\begin{align*}
    \iota(\rho_1(\P)) &=\iota( \langle g \in \cohom(\K_1) \text{ homogeneous}\mid \cone(g) \not \in \P \rangle)\\
    &=\langle f \in \cohom(\K_2) \text{ homogeneous}\mid f \otimes_2 1_{\unit_1} \in \langle g \in \cohom(\K_1) \text{ homogeneous}\mid \cone(g) \not \in \P \rangle\rangle\\
    &=\langle f \in \cohom(\K_2) \text{ homogeneous}\mid \cone(f \otimes_2 1_{\unit_1}) \not \in \P\rangle\\
    &=\langle f \in \cohom(\K_2) \text{ homogeneous}\mid \cone(f) \otimes_2 \unit_1 \not \in \P\rangle. 
\end{align*}
The last equality follows because tensor product is a triangulated functor. On the other hand, traversing the diagram first down and then right yields
\begin{align*}
    \rho_2(\eta(\P))&=\rho_2(\{x \in \K_2 \mid x \otimes_2 \unit_1 \in \P\})\\
    &= \langle f \in \cohom(\K_2) \text{ homogeneous} \mid \cone(f) \not \in \{x \in \K_2 \mid x \otimes_2 \unit_1 \in \P\} \rangle \\
    &=\langle f \in \cohom(\K_2) \text{ homogeneous}\mid \cone(f) \otimes_2 \unit_1 \not \in \P\rangle.
    \end{align*}
\end{proof}

\subsection{Conservative right tensor-invertible actions on unipotent tensor categories}
\label{subsect:unip-f}

In this section, we assume that $\widehat{\A}$ and $\widehat{\K}$ are two rigidly-compactly generated monoidal triangulated categories with respective compact parts $\A$ and $\K$, that there is an action of $\widehat{\K}$ on $\widehat{\A}$, and that $\A = \E(\A)$, i.e.~ $\A$ is unipotent. In subsequent sections, we will apply these results to the specific cases where $\A$ is $\umod(A)$ for $A$ a finite dimensional unipotent Hopf algebra, and $\K$ is either $\ulrp(A^{\env})$ or its shell $\E$. 

 Throughout this section, the following two properties of actions will be important.

\begin{defn}
\label{defn:faithful-right-inv}
The action of $\widehat{\K}$ on $\widehat{\A}$ is 
\begin{enumerate}
\item {\emph{conservative}} 
if $x * \unit \cong 0$ implies $x \cong 0$ for all $x \in \widehat{\K}$, and
\item {\emph{right invertible}} if the functor $\K \to \A$ defined by $x \mapsto x* \unit$ for all $x \in \K$ has a right inverse $\A \to \K$, and {\emph{right tensor-invertible}} if it is right invertible such that the right inverse is a monoidal functor. 
\end{enumerate}
\end{defn}

We will typically denote by $G$ the functor $\widehat{\K} \to \widehat{\A}$ defined by $x \mapsto x*\unit$. By abuse of notation, we often use $G$ as well to denote this functor when restricted to $\K$. The condition that $G$ is right tensor-invertible is that there is a monoidal functor $F\colon \A \to \K$ 
with $GF \cong 1_{\A}$. Note that we do not require $G$ itself to be a monoidal functor. 

We first show that under sincerity, thick right ideals of $\K$ can be recovered from their images under $G$.  

\begin{prop}
    \label{prop:recover-thick-t}
Assume that the action of $\widehat{\K}$ on $\widehat{\A}$ is conservative. Let $\I$ be a thick right ideal of $\K$. Suppose $x \in \K$ satisfies $x*\unit \in \Thick(y* \unit \mid y \in \I).$ Then $x \in \I.$
\end{prop}

\begin{proof}
By Theorem \ref{thm:Rickard-props}(a), it suffices to show that $L_{\I} x \cong 0$.
  Note that $L_{\I}x \cong L_{\I} \unit \otimes x$ by Theorem \ref{thm:Rickard-props}(b). We have $L_{\I} \unit \otimes y \cong L_{\I} y \cong 0$ for all $y \in \I$, and therefore $L_{\I} \unit * (y * \unit) \cong 0$ for all $y \in \I$. 
This implies that $L_{\I} \unit *(x * \unit) \cong (L_{\I} \unit \otimes x)* \unit \cong 0$
since $x * \unit$ is in the thick subcategory generated by $y*\unit$, over all $y \in \I$.
By the sincerity assumption, $L_{\I} \unit \otimes x \cong L_{\I} x \cong 0$, and hence $x \in \I$. 
\end{proof}

\begin{cor}
\label{cor:right-det-by-e}
Assume that the action of $\widehat{\K}$ on $\widehat{\A}$ is conservative and right invertible with right inverse $F$. Let $\I$ be a thick right ideal of $\K$. Then
\[
\I = \langle F(y*\unit) \mid y \in \I \rangle_r.
\]
\end{cor}

\begin{proof}
  Both containments are readily verified using Proposition \ref{prop:recover-thick-t} and right invertibility. 
  For instance, to show that
    \[
    \I \subseteq \langle F(y*\unit)\mid y \in \I\rangle_r =: \I'',
    \]
    it is enough by Proposition \ref{prop:recover-thick-t} to show that if $x \in \I$, then $x* \unit \in \Thick(y * \unit \mid y \in \I'').$ Indeed, if $x \in \I$, then $z:=F(x * \unit) \in \I''$, and 
   $
    x * \unit \cong F(x* \unit) * \unit=z*\unit
   $
    since $F$ is a right inverse to $G$. 
\end{proof}

In view of classifying $\K$-submodule categories of $\A$ using the functor $G$, we next note a useful lemma which uses the unipotence of $\A$ in a critical way. This lemma does not require the action to be conservative or right invertible. 

\begin{lem}
    \label{lem:ginv-ideal}
     Let $\J$ be a thick subcategory of $\A$. Then $G^{-1}(\J)$ is a thick right ideal of $\K$. If $\J$ is a thick $\K$-submodule category, then $G^{-1}(\J)$ is a thick two-sided ideal of $\K$.
\end{lem}

\begin{proof}
First, suppose that $\J$ is a thick subcategory of $\A$ and $x \in G^{-1}(\J)$. Let $y \in \K$; we claim that $x \otimes y \in G^{-1}(\J)$, that is, $G(x \otimes y) \cong x *G(y) \in \J$. Consider the collection of objects
\[
\J':=\{a \in \A \mid x * a \in \J \}. 
\]
We know that $\unit \in \J'$, since $G(x) \in \J$ by assumption. But $\J'$ is a thick subcategory of $\A$, since the action is a triangulated functor and $\J$ is a thick subcategory. Since $\Thick(\unit)= \A$, it follows that $\J' = \A$, that is, $x*a \in \J$ for all $a \in \A$, and hence $x*G(y) \in \J$. Therefore, $G^{-1}(\J)$ is indeed a right ideal. 

Now assume $\J$ is a thick submodule category. If $x \in G^{-1}(\J)$ and $y \in \K$, then 
\[
G(y \otimes x)=(y \otimes x)* \unit \cong y *(x*\unit) = y*G(x).
\]
Since $G(x) \in \J$ and $\J$ is a $\K$-submodule category, it follows that $y \otimes x \in G^{-1}(\J)$, and so $G^{-1}(\J)$ is a left ideal, hence two-sided. 
\end{proof}

We are now ready to state the correspondence between thick ideals of $\K$ and thick $\K$-submodules of $\A$.

\begin{prop}
    \label{prop:ideal-bij-spc-general}
    Assume that the action of $\widehat{\K}$ on $\widehat{\A}$ is conservative and right invertible. Then $G$ induces a bijection between thick ideals of $\K$ and thick $\K$-submodule categories of $\A$ via the maps
    \begin{center}
        \begin{tikzcd}
\I \arrow[rr, maps to]                              &  & G(\I)                            \\
\Ideals(\K) \arrow[rr, bend left] &  & \Submods_{\K}(\A) \arrow[ll, bend left] \\
G^{-1}(\J)                                          &  & \J \arrow[ll, maps to]                             
\end{tikzcd}
    \end{center}
\end{prop}

\begin{proof}
Let $\I$ be a thick ideal of $\K$ and $x \in \I$. Note that since 
\[
y* G(x) = y*(x*\unit) \cong (y \otimes x)*\unit = G(y \otimes x)
\]
for any $y \in \K$, and since $y \otimes x \in \I$, it follows that $G(\I)$ is a $\K$-submodule category of $\A$. On the other hand, if $\J$ is a $\K$-submodule category of $\A$, then $G^{-1}(\J)$ is a two-sided ideal of $\K$ by Lemma \ref{lem:ginv-ideal}. Hence the maps in the proposition are well defined. 

If $\J$ is a thick $\K$-submodule category of $\A$, then clearly $G(G^{-1}(\J)) \subseteq \J$. On the other hand, if $a \in \J$ and if we denote the right inverse of $G$ by $F$, then since $F(a) \in G^{-1}(\J)$, it is clear that $G(G^{-1}(\J)) = \J$.

Let $\I$ be a two-sided thick ideal of $\K$. Clearly $\I \subseteq G^{-1}(G(\I))$. For the opposite containment, suppose that $G(x) \in G(\I)$. Since $\I$ is in particular a right ideal, by Proposition \ref{prop:recover-thick-t}, it follows that $x \in \I$, and so $G^{-1}(G(\I)) = \I.$
\end{proof}

\begin{cor}
\label{cor:spck-vs-spca-unip}
    Assume that the action of $\widehat{\K}$ on $\widehat{\A}$ is conservative and right invertible. If $\K = \E(\K)$, then $\Spc (\K) \cong \Spc (\A)$ via the map of Proposition \ref{prop:ideal-bij-spc-general}.
\end{cor}

\begin{proof}
    Indeed, in this case all thick subcategories of $\K$ are thick ideals, and all thick subcategories of $\A$ are thick submodule categories, so that Proposition \ref{prop:ideal-bij-spc-general} gives a bijection between the thick ideals of $\K$ and the thick ideals of $\A$. The result now follows from the more general Proposition \ref{prop:ideals-homeo} below, which we record for future reference.
\end{proof}

\begin{prop}
    \label{prop:ideals-homeo}
    Let $G' \colon \K' \to \A'$ be a (not necessarily monoidal) triangulated functor between rigid monoidal triangulated categories that induces a bijection on thick ideals via
    \begin{center}
        \begin{tikzcd}
\I \arrow[rr, maps to]                               &  & \langle G'(x) \mid x \in \I\rangle                              \\
\Ideals(\K') \arrow[rr, bend left] &  & \Ideals(\A') \arrow[ll, bend left] \\
G'^{-1}(\J)                                           &  & \J \arrow[ll, maps to]                              
\end{tikzcd}
\end{center}
Then these maps restricted to $\Spc (\A')$ and $\Spc (\K')$ define inverse homeomorphisms. 
\end{prop}

\begin{proof}
    Recall that $\P$ is prime if and only if $\I \otimes \J \subseteq \P$ implies that $\I$ or $\J \subseteq \P$. Recall also that by Lemma \ref{lem:tensor-int}, $\I \cap \J = \langle \I \otimes \J \rangle$, so the prime condition can be phrased: if $\I \cap \J \subseteq \P$, then $\I$ or $\J \subseteq \P$. 
    
    We must show that $\P \in \Spc (\A')$ if and only if $G'^{-1}(\P) \in \Spc (\K')$. For one direction, suppose $G'^{-1}(\P)$ is prime and $\I \cap \J \subseteq \P$. Since $G'^{-1}(\I \cap \J) = G'^{-1}(\I) \cap G'^{-1}(\J)$, we have $G'^{-1}(\I) \cap G'^{-1}(\J) \subseteq G'^{-1}(\P)$, and by the prime assumption $G'^{-1}(\I)$ or $G'^{-1}(\J) \subseteq G'^{-1}(\P)$; assume $G'^{-1}(\I)\subseteq G'^{-1}(\P)$. Since we have assumed that $G'$ induces a bijection on thick ideals, this implies $\I \subseteq \P$, and $\P$ is prime.

    On the other hand, suppose $\P$ is prime. By the assumption of the proposition, to show that $G'^{-1}(\P)$ is prime, it is enough to check that if $G'^{-1}(\I) \cap G'^{-1}(\J) \subseteq G'^{-1}(\P)$, then $G'^{-1}(\I)$ or $G'^{-1}(\J) \subseteq G'^{-1}(\P)$. This implies that $\langle G'(x) \mid x \in G'^{-1}(\I) \cap G'^{-1}(\J) \rangle = \I \cap \J \subseteq \P.$ By primeness of $\P$, either $\I$ or $\J \subseteq \P$, and we are done.

    Hence $\Spc (\A')$ and $\Spc (\K')$ are in bijection. Since the maps on thick ideals are containment preserving, it is straightforward to check that the induced map on spectra is continuous.
\end{proof}

Our next goal is to understand better the relationship between $\Spc(\K)$ and $\Spc(\A)$ when $\K$ is {\emph{not}} unipotent. To define a map between these spaces, we use the language of abstract support data. Recall the following definition from~\cite[Definition 4.1.1]{NVY1}.

\begin{defn}\label{defn:support-datum}
A {\emph{support datum}} on a 
\mtc\ $\K$ is a map $\sigma$ from $\K$ to the collection of closed subsets of a 
topological space $X$ for which the following properties hold: 
    \begin{enumerate}
    \item $\sigma(0)=\varnothing$ and $\sigma(\unit)= X$;
    \item $\sigma(x \oplus y)=\sigma(x) \cup \sigma(y)$ for all $x,y \in \K$;
    \item $\sigma(\Sigma x)=\sigma(x)$ for all $x \in \K$;
    \item if $x \to y \to z \to \Sigma x$ is a distinguished triangle in $\K$, then $\sigma(x) \subseteq \sigma(y) \cup \sigma(z);$
    \item $\sigma(x) \cap \sigma(z) = \bigcup_{y \in \K} \sigma(x \otimes y \otimes z)$ for all $x, z \in \K$.
    \end{enumerate}
\end{defn}

\begin{rem}
\label{rem:tpp}
Note that when $\K$ is braided or thickly generated by its unit, Definition \ref{defn:support-datum}(v) reduces to the more familiar identity
\begin{enumerate}
\item[(v$'$)] $\sigma(x) \cap \sigma(z) = \sigma(x \otimes z)$. 
\end{enumerate}
We will refer to Definition \ref{defn:support-datum}(v) as the {\emph{noncommutative tensor product property}} and to (v$'$) as the {\emph{tensor product property}}. Note that any support datum satisfying the tensor product property automatically satisfies the noncommutative tensor product property.
\end{rem}

\begin{prop}
    \label{prop:support-funct-g}
    Assume that the action of $\widehat{\K}$ on $\widehat{\A}$ is conservative and right tensor-invertible. For all $x\in \K$, let
    \[
    \sigma(x):= \bigcup_{x' \in \langle x \rangle} \supp(G(x')) . 
    \]
    Then $\sigma$ is a support datum on $\K$ with values in a space $X$ whose points coincide with $\Spc (\A)$ but with a different topology. If $\K$ has a thick generator, then $\sigma$ is a support datum with values in $\Spc(\A)$. 
\end{prop}

\begin{proof}
Note that $\sigma(x)$ is by definition a union of closed subsets of $\Spc(\A)$. Define $X$ to be the space with the same points as $\Spc(\A)$, but where closed sets are defined to be arbitrary intersections of sets of the form $\sigma(x)$, over all $x \in \K$. Then by definition $\sigma(x)$ is a closed subset of $X$. If $\K$ has a thick generator $g$, then $\sigma(x)$ is already closed in $\Spc(\A)$ with the usual topology since
\[
\bigcup_{x' \in \langle x \rangle} \supp(G(x')) = \supp(G(g \otimes x \otimes g)),
\]
 so in that case $\sigma$ as defined is a support datum with values in $\Spc(\A)$.

Properties (i) and (iii) of a support datum are straightforward to check. 
The containment $\sigma(x \oplus y) \supseteq \sigma(x) \cup \sigma(y)$ of (ii) is also straightforward, since $x$ and $y$ are in $\langle x \oplus y \rangle$, and the containment $\sigma(x \otimes z) \subseteq \sigma(x) \cap \sigma(z)$ of (v) is similar. 

For the remainder of the proof, we will use the notation $G \colon \widehat{\K} \rightarrow \widehat{\A}$
for the functor $G (x) = x * e$ as before, and denote by $F$ a right inverse to $G$.

        The containment $\subseteq$ of (ii) will follow from (iv) and (iii), since there is always a distinguished triangle of the form $x \oplus y \to y \to \Sigma x \to \Sigma(x \oplus y)$. We now show (iv). Suppose that $x \to y \to z \to \Sigma x$ is a distinguished triangle in $\K$ and $\P \not \in \sigma(y) \cup \sigma(z)$. We must show $\P \not \in \sigma(x).$ Since $\P \not \in \sigma(y) \cup \sigma(z)$, for all $y' \in \langle y \rangle$ and for all $z' \in \langle z \rangle,$ we have $G(y') \in \P$ and $G(z') \in \P$. We now claim that
    \[
    \I:=\{ x' \in \langle x \rangle \mid G(y' \otimes x' \otimes z') \in \P \;\; \forall \; \; y',z' \in \K \}
    \]
    is equal to $\langle x \rangle$. On one hand, by definition $\I \subseteq \langle x \rangle$, and it is straightforward to verify that $\I$ is a thick ideal. But $x \in \I$, since we have a distinguished triangle
    \[
    G(y' \otimes x \otimes z') \to G(y' \otimes y \otimes z') \to G(y' \otimes z \otimes z') \to \Sigma(G(y' \otimes x \otimes z'))
    \]
    for any $y', z' \in \K$, and the second two objects are in $\P$ by the assumption on $y$ and $z$. Since $\P$ is triangulated, $G(y' \otimes x \otimes z') \in \P$ for any $y', z' \in \K$, that is, $x \in \I$. Hence $\I = \langle x \rangle$. This immediately implies that $G(x') \in \P$
for any $x' \in \langle x \rangle$, that is, $\P \not \in \sigma(x)$.

    We now show the reverse containment of (v), that is, 
    \[
    \sigma(x) \cap \sigma(z) \subseteq \bigcup_{y \in \K} \sigma(x \otimes y \otimes z).
    \]
    Suppose $\P \in \sigma(x) \cap \sigma(z)$. That is, for some $x' \in \langle x \rangle$ and some $z' \in \langle z \rangle$, the objects $G(x')$ and $G(z')$ are not in $\P$. Since $\P$ is prime in $\A$, this means that $G(x') \otimes G(z') \not \in \P$, and hence $F(G(x') \otimes G(z')) \cong FG(x') \otimes FG(z')$ is an object of $\K$ whose image under $G$ is not contained in $\P$. To complete the proof of this containment, it now suffices to observe that $FG(x') \otimes FG(z') \in \langle x \otimes y \otimes z\rangle$ for some $y \in \K$. By \cite[Lemma 3.1.2]{NVY1}, we know that $FG(x') \otimes FG(z') \in \langle x \otimes y \otimes z \mid y \in \K \rangle$, since by Corollary \ref{cor:right-det-by-e}, $FG(x') \in \langle x \rangle$ and $FG(z') \in \langle z \rangle$. Therefore, there must exist a finite collection $y_1, \ldots, y_n$ of objects such that $FG(x') \otimes FG(z') \in \langle x \otimes y_i \otimes z \mid i = 1, \ldots, n\rangle.$ In particular, this implies that 
    \[
    FG(x') \otimes FG(z') \in \left \langle x \otimes \left ( \bigoplus_{i=1}^n y_i \right )  \otimes z \right \rangle,
    \]
    and so the claim is proven. 
\end{proof}

The point of defining a support datum above is that we can use it to give a continuous map on Balmer spectra; in particular, we deduce the following proposition.

\begin{prop}
    \label{prop:map-on-spectra-t-k}
    Assume that the action of $\widehat{\K}$ on $\widehat{\A}$ is conservative and right tensor-invertible. Then there is a set map of the form
    \begin{align*}
        \Spc (\A) & \xrightarrow{\xi} \Spc (\K) \;\;\\
        \P & \mapsto \{ x \in \K \mid \P \not \in \supp(G(x')) \;\; \forall \;\; x' \in \langle x \rangle\}\\
        &\;\;\;\;\;\;\; =\{x \in \K \mid G(x')  \in \P \;\; \forall \;\; x' \in \langle x \rangle\}.
    \end{align*}
    If $\K$ has a thick generator, then $\xi$ is continuous. If either $\Spc( \K)$ or $\Spc (\A)$ is Noetherian, then $\xi$ is surjective.
\end{prop}

\begin{proof}
    The existence and formula for the map $\xi$ are directly by \cite[Theorem 4.2.2]{NVY1} and Proposition \ref{prop:support-funct-g}, as is the fact that $\xi$ is continuous if $\K$ has a thick generator. It remains to show the surjectivity claims. Let $\Q \in \Spc (\K)$; we will show that under either Noetherianity assumption, there exists some $\P \in \Spc (\A)$ such that $\xi(\P)=\Q$. 

    First, suppose that $\Spc (\K)$ is Noetherian. By Theorem \ref{thm:gen-idealsclass}(c), there exists some object $x$ in $\K$ such that $\supp(x) = \overline{\{\Q\}}$. This $x$ consequently has the property that $x \not \in \Q$, and if $x'$ is another object with $x' \not \in \Q$, then $x \in \langle x' \rangle$. Consider the thick ideal $G(\Q)$ in $\A$. By Proposition \ref{prop:ideal-bij-spc-general}, $G(x) \not \in G(\Q)$. Since $\A$ is rigid, and so all thick ideals are intersections of prime ideals (by \cite[Proposition 4.1.1]{NVY2}), there exists some prime ideal $\P$ in $\Spc (\A)$ such that $G(\Q) \subseteq \P$, and $G(x) \not \in \P$. We have $\Q \subseteq \xi(\P)$, since if $y \in \Q$ then $y' \in \Q$ for all $y' \in \langle y \rangle$, and so since $\P$ contains $G(\Q)$ it contains $G(y')$ for all such $y'$. On the other hand, if $x' \not \in \Q$, then $x' \not \in \xi(\P)$, since $x$ is in $\langle x' \rangle$ and $G(x) \not \in \P$. Therefore, $\xi(\P)=\Q$. 

    Now suppose instead that $\Spc (\A)$ is Noetherian. Pick $x_1 \not \in \Q$. As before, we know that $G(x_1)\not \in G(\Q)$, and therefore there exists $\P_1 \in \Spc (\A)$ such that $G(\Q) \subseteq \P_1$ and $G(x_1) \not \in \P_1$. Now suppose that there exists some $x_2 \not \in \Q$ such that for all $x_2' \in \langle x_2\rangle$, we have $G(x_2') \in \P_1$ (in other words, $x_2 \not \in \Q$ and $x_2 \in \xi(\P_1)$). Since $\Q$ is prime and neither
$x_1$ nor $x_2$ is in $\Q$, there exists $y_1 \in \K$ such that $x_1 \otimes y_1 \otimes x_2 \not \in \Q$. 
Note $y_1 \otimes x_2$ also satisfies the property that for any object $y_1' \in \langle y_1 \otimes x_2\rangle,$ we have $G(y_1') \in \P_1$, and $y_1 \otimes x_2 \not \in \Q$. Therefore, without loss of generality, we could have chosen $y_1 \otimes x_2$ for $x_2$, that is, we have $x_2$ with $G(x_2') \in \P$ for all $x_2' \in \langle x_2 \rangle$ and $x_1 \otimes x_2 \not \in \Q$. Now since $G(x_1 \otimes x_2) \not \in G(\Q)$, there exists $\P_2 \in \Spc (\A)$ such that $G(\Q) \subseteq \P_2$ and $G(x_1 \otimes x_2) \not \in \P_2$. 

    Continue inductively in this manner. At step $m$, we have chosen $x_1, \ldots, x_m$ such that each $x_i \not \in \Q$, the tensor product $x_1 \otimes \cdots \otimes x_m \not \in \Q$, and we also have $\P_1,\ldots, \P_m \in \Spc (\A)$ such that for each $i$, the prime $\P_i$ contains $G(\Q)$, does not contain $G(x_1 \otimes \cdots \otimes x_i)$, and $G(x_{i+1}') \in \P_i$ for all $x_{i+1}' \in \langle x_{i+1} \rangle$. Note in particular that $G(x_1 \otimes x_2 \otimes \cdots \otimes x_{i+1}) \in \P_i$. If there exists some $x_{m+1} \not \in \Q$ such that $G(x_{m+1}') \in \P_m$ for all $x_{m+1}' \in \langle x_{m+1} \rangle$, then (possibly by relabeling $x_{m+1}$ by tensoring on the left by some other object) we have that $x_1 \otimes\cdots \otimes x_m \otimes x_{m+1} \not \in \Q$ and $G(x_{m+1}') \in \P_m$ for all $G(x_{m+1}') \in \langle x_{m+1} \rangle$. In this case, we can then continue the process, selecting some $\P_{m+1} \in \Spc (\A)$ such that $G(\Q) \subseteq \P_{m+1}$ and $G(x_1 \otimes \cdots \otimes x_{m+1}) \not \in \P_{m+1}$. 

    We claim that this process gives rise to a descending chain of closed subsets of $\Spc (\A)$:
    \[
    \supp(G(x_1)) \supsetneq \supp(G(x_1 \otimes x_2)) \supsetneq \ldots \supsetneq \supp(G(x_1 \otimes x_2 \otimes \cdots \otimes x_m)) \supsetneq \ldots
    \]
Since $G(x_1 \otimes \cdots \otimes x_i)$ is in the thick subcategory generated by 
$G(x_1 \otimes \cdots \otimes x_{i-1})$, Lemma \ref{lem:ginv-ideal} implies      
that $\supp(G(x_1 \otimes \cdots \otimes x_i)) \subseteq \supp(G(x_1 \otimes\cdots \otimes x_{i-1}))$. 
The containment is strict because $\P_i$ is in $\supp(G(x_1 \otimes\cdots \otimes x_i))$ and is not in $\supp(G(x_1 \otimes \cdots \otimes x_{i+1})).$ 

    Now by the Noetherianity assumption on $\Spc (\A)$, this chain must terminate. In particular, there is a step, say $n$, for which there does not exist any $x_{n+1} \in \Q$ such that $G(x_{n+1}') \in \P_n$ for all $x_{n+1}' \in \langle x_{n+1} \rangle$. Set $\P:=\P_n$. Since $G(\Q) \subseteq \P$, by the formula for $\xi(\P)$ we have $\Q \subseteq \xi(\P)$. On the other hand, if $x \not \in \Q$, then there must exist some $x' \in \langle x \rangle$ such that $G(x') \not \in \P$, so $x \not \in \xi(\P)$. That is, $\xi(\P)=\Q$. 
    \end{proof}

We highlight next one useful consequence of surjectivity in the proposition: 
Noetherianity of $\Spc (\A)$ implies Noetherianity of $\Spc (\K)$ in the case that $\xi$ is continuous. Recall that Noetherianity of Balmer spectra is particularly important, since it implies a bijection between thick ideals and specialization-closed subsets of the Balmer spectrum (Theorem \ref{thm:gen-idealsclass}(c)); note that in a Noetherian topological space specialization-closed and Thomason subsets are the same.

    \begin{cor}
        \label{cor:spc-t-noeth}
        Assume that $\K$ has a thick generator, that the action of $\widehat{\K}$ on $\widehat{\A}$ is conservative and right tensor-invertible, and assume that $\Spc (\A)$ is Noetherian. Then $\Spc (\K)$ is Noetherian.
    \end{cor}

\begin{proof}
    By Proposition \ref{prop:map-on-spectra-t-k}, there is a continuous surjective map $\Spc (\A) \to \Spc (\K)$. 
\end{proof}

\subsection{Stevenson support for conservative right invertible actions}

We continue with our assumption that $\widehat{\K}$ is a rigidly-compactly generated \mtc{} acting on a unipotent rigidly-compactly generated \mtc{} $\widehat{\A}$. When the action is conservative and right invertible, we give a characterization of the Stevenson module-theoretic support on $\A$, and show that the action of $\K$ on $\A$ is parametrizing in the sense of Definition \ref{defn:param}.

\begin{prop}
\label{prop:stev-supp}
    Assume that the action of $\widehat{\K}$ on $\widehat{\A}$ is conservative and right invertible with right inverse $F$. Assume also that $\Spc (\K)$ is Noetherian. Then for any $a \in \A$, the Stevenson module-theoretic support for $a$ is given by
    \[
     \supp^*(a) = \supp(F(a)).
    \]
    In particular, the action of $\K$ on $\A$ is parametrizing, that is, the Stevenson module-theoretic support gives a bijection between the thick submodule categories of $\A$ and the specialization-closed subsets of $\Spc (\K)$.
\end{prop}

\begin{rem}
    The Noetherianity assumption on $\Spc (\K)$ in this proposition is necessary for the Stevenson module-theoretic support, which is based on the Balmer--Favi support in $\K$, to be a well-defined support theory (see \cite{CV}).
\end{rem}

\begin{proof}
    For the first statement, we simply check:
    \begin{align*}
        \supp^*(a) &=\{ \P \in \Spc (\K) \mid {\it{\Gamma}}_{\P} \unit * a \not \cong 0 \}\\
        &= \{ \P \in \Spc (\K) \mid ({\it{\Gamma}}_{\P} \unit \otimes F(a))* \unit \not \cong 0 \}\\
                &= \{ \P \in \Spc (\K) \mid {\it{\Gamma}}_{\P} \unit \otimes F(a) \not \cong 0 \}\\
                 &= \{ \P \in \Spc (\K) \mid \P \in \supp(F(a)) \}\\
                 &=\supp(F(a)).
    \end{align*}

    For the second statement, first note that for any $a \in \A$,
    $
    G^{-1}(\langle a \rangle) = \langle F(a) \rangle.
    $
    Clearly, $F(a) \in G^{-1}(\langle a \rangle)$, so the containment $\supseteq$ of the statement follows. On the other hand, if $x \in G^{-1}(\langle a \rangle)$, then by Corollary \ref{cor:right-det-by-e},
    \[
    \langle x \rangle = \langle FG(y) \mid  y \in \langle x \rangle \rangle_r.
    \]
Now $G(y) \in \langle a \rangle$ for all $y \in \langle x \rangle$, 
and therefore $FG(y) \in \langle F(a) \rangle$. Hence $x$ is in $\langle F(a) \rangle$.

Recall from Section \ref{subsect:RickStevenson} that for the action of $\K$ on $\A$ to be parametrizing, the maps $\Phi$ and $\Theta$ should be inverse bijections, where
\begin{align*}
\Phi(\J) &:= \bigcup_{a \in \J} \supp^*(a) \text{ for any }\K\text{-submodule category } \J,\\
\Theta(S)&:= \{a \in \A \mid \supp^*(a) \subseteq S\} \text{ for any specialization-closed subset }S\text{ of }\Spc (\K).
\end{align*}
We check now that for any $\K$-submodule category $\J$ of $\A$, we have
\begin{align*}
\Theta(\Phi(\J)) &= \Theta \left (\bigcup_{a \in \J} \supp^*(a) \right )\\
&=\Theta\left (\bigcup_{a \in \J} \supp(F(a)) \right)\\
&=\{b \in \A \mid \supp(F(b)) \subseteq \bigcup_{a \in \J} \supp(F(a)) \}\\
&=\{b \in \A \mid F(b) \in \langle F(a) \mid a \in \J \rangle \}\\
&=\J.
\end{align*}
The fourth equality follows from the general classification theorem for monoidal triangulated categories (Theorem \ref{thm:gen-idealsclass}). The fifth equality follows from the fact proven above that $G^{-1}(\langle a \rangle) = \langle F(a) \rangle$, since if $F(b) \in \langle F(a) \mid a \in \J \rangle$, then $F(b) \in G^{-1}(\J)$, and by Proposition \ref{prop:ideal-bij-spc-general} this implies that $G(F(b)) \cong b$ is in $\J$. 

On the other hand, if $S$ is a specialization-closed subset of $\Spc (\K)$, then
\begin{align*}
\Phi(\Theta(S)) &= \Phi( \{a \in \A \mid \supp^*(a) \subseteq S \} )\\
&= \bigcup_{a \text{ with } \supp^*(a) \subseteq S} \supp^*(a)\\
&= \bigcup_{a \text{ with }\supp(F(a)) \subseteq S} \supp(F(a))\\
&=S.
\end{align*}
To see the last equality, note that again by Theorem \ref{thm:gen-idealsclass}, given any closed subset $S$ of $\Spc (\K)$ there is an object $x$ with $\supp(x)=S$. By Corollary \ref{cor:right-det-by-e}, $\langle x \rangle = \langle FG(x) \rangle$ and hence $\supp(x) = \supp(FG(x))$ for any $x \in \K$. 
\end{proof}

\begin{rem}
One might attempt to use the Stevenson support on $\A$, combined with the universal property of the Balmer support on $\Spc(\A)$, to obtain a map $\Spc(\K) \to \Spc(\A)$ as a possible candidate for an inverse map to the map $\xi$ from Proposition \ref{prop:map-on-spectra-t-k}. Note however that the Stevenson support will only satisfy the tensor product property (see Remark \ref{rem:tpp}) if $\K$ is completely prime. 
\end{rem}

\section{Unipotent Hopf algebras}\label{sec:unipotent-Hopf}
\label{sec:Hopf}

In this section, we will apply our previous results to the specific case of finite dimensional unipotent Hopf algebras. 
Throughout this section we fix a finite dimensional unipotent Hopf algebra $A$ over a field $k$;
in particular, $A/\rad(A)\cong k$. 
Note that $A^{\env}$ is also a unipotent Hopf algebra:
$\rad(A)\ot_k A + A\ot_k\rad(A)$ is a nilpotent ideal of $A^{\env}$, and a dimension argument
shows that the quotient of $A^{\env}$ by this ideal is isomorphic to $A/\rad(A)\ot_k A/\rad(A)
\cong k\ot_k k\cong k$. 
It follows that $\rad(A^{\env})$ is precisely that ideal, and $A^{\env}$ is unipotent. 
Set $\E = \Thick(A)$, the shell of the left-right projective category $\ulrp(A^{\env})$ of $A$-bimodules. 
A few needed results will be stated more generally for a finite dimensional algebra $\Lambda$.

\subsection{A subgroup of the Picard group of $\ulrp(\Lambda^{\env})$}

The Picard group of a monoidal category is the collection of tensor-invertible objects (see~\cite{BalmerPic2010}). 
We recall some results on $\Lambda$-bimodules 
that are tensor-invertible with respect to $\otimes_{\Lambda}$,
paraphrasing part of~\cite[Prop.~5.2]{Bass1968}. If $\alpha$ is an algebra automorphism of $\Lambda$, and $M$ is a left $\Lambda$-module, write ${_{\alpha}}M$ for the left $\Lambda$-module that is $M$ as a vector space and has action 
\[
a.m:=\alpha(a)m,
\]
where $\alpha(a)m$ is the action of $\alpha(a)$ on $m$ in $M$. We use similar notation for right modules and for bimodules. So, for example, if $B$ is a $\Lambda$-bimodule and $\alpha$ and $\beta$ are algebra automorphisms of $\Lambda$, then ${_{\alpha}} B_{\beta}$ is the new bimodule such that $\Lambda$ acts through $\alpha$ on the left and through $\beta$ on the right.

\begin{prop}[Bass~\cite{Bass1968}] 
\label{prop:Bass}  
Let $\alpha, \beta,\gamma$ be algebra automorphisms of $\Lambda$.
Then: 
\begin{enumerate}[\qquad \rm(a)]
\item \ ${}_{\alpha} \Lambda_{\beta} \cong {}_{\gamma \alpha}\Lambda_{\gamma\beta}$
as $\Lambda$-bimodules.
\item \ $({}_1\Lambda_{\alpha})\otimes_{\Lambda} ({}_1 \Lambda_{\beta})\cong {}_1\Lambda_{\alpha\beta}$
as $\Lambda$-bimodules.
\item \ ${}_1\Lambda_{\alpha}\cong {}_1 \Lambda _1$ as $\Lambda$-bimodules 
if, and only if,
$\alpha$ is an inner automorphism.
\end{enumerate}
\end{prop}

The proof of part (a) in the proposition is that  
the algebra automorphism $\gamma\colon \Lambda\rightarrow \Lambda$
is in fact itself a $\Lambda$-bimodule isomorphism
from $ {}_{\alpha}\Lambda_{\beta}$ to ${}_{\gamma\alpha}\Lambda_{\gamma\beta}$.
As a consequence of part (a) in the proposition, each bimodule
$ {}_{\alpha}\Lambda_{\beta}$ is left-right projective (by taking
$\gamma = \alpha^{-1}$, respectively $\gamma = \beta^{-1}$).
As a consequence of parts (a) and (b), each 
bimodule ${}_{\alpha}\Lambda_{1}$ is invertible, with inverse
${}_{\alpha^{-1}}\Lambda_{1}$. 
The collection of $\Lambda$-bimodules ${}_{\alpha}\Lambda_{1}$, for
outer automorphisms $\alpha$, is isomorphic to a subgroup 
of the Picard group of $\ulrp(\Lambda^{\env})$.

We also note that $\Aut(A)$ acts on $\Spc(\umod(A))$ via twisting in the unipotent case. Note that if $A$ is a non-unipotent finite-dimensional Hopf algebras, then one gets an action of the group of Hopf algebra automorphisms on $\Spc(\K)$ in a standard way, see \cite[Example 3.3]{HuangVashaw2025}; however, here $\Aut(A)$ are just algebra automorphisms, not necessarily Hopf algebra automorphisms, and we need the unipotent assumption. 

\begin{prop}
Let $A$ be a finite-dimensional unipotent Hopf algebra. Then $\Aut(A)$ acts on $\Spc(\umod(A))$ in the following way: for $\alpha \in \Aut(A)$ and $\P \in \Spc(\umod(A))$, set $\alpha.\P = \{ {}_{\alpha} M \mid M \in \P\}.$ 
\end{prop}

\begin{proof}
We need to justify that $\alpha.\P$ is again a point of $\Spc(\umod(A))$, that is, it is prime. Recall that by Lemma \ref{lem:tensor-int}, since $\umod(A)$ is unipotent, $\I \cap \J = \langle \I \otimes \J \rangle$, so to show that $\alpha.\P$ is prime we must show that $\I \cap \J \subseteq \alpha.\P$ implies $\I$ or $\J \subseteq \alpha.\P$. But indeed, if $\I \cap \J \subseteq \alpha.\P$, this means that ${}_{\alpha^{-1}} M \in \P$ for all $M \in \I \cap \J$. This implies $(\alpha^{-1}. \I) \cap (\alpha^{-1}.\J) \subseteq \P$, where $\alpha^{-1}.\I$ is the thick ideal consisting of ${}_{\alpha^{-1}} M$ for all $M \in \I$, and likewise for $\J$. Since $\P$ is prime, this implies that either $\alpha^{-1}.\I$ or $\alpha^{-1}.\J$ is contained in $\P$, which means that either $\I$ or $\J$ is contained in $\alpha. \P$. Thus $\alpha.\P$ is in $\Spc(\umod(A))$. 
\end{proof}

\subsection{The Balmer spectra of $\E$ and $\ulrp(A^{\env})$}

We first characterize $\Spc (\E)$ in two ways: in terms of the spectrum of $\umod(A^{\env})$, and in terms of the spectrum of $\umod(A)$. 
Recall that $\umod(A^{\env})$ is a \mtc\ via the tensor product over $k$ using the Hopf algebra structure. The subcategories $\E$ and $\ulrp(A^{\env})$ of $\umod(A^{\env})$ are triangulated, but are not monoidal subcategories, since the tensor product used in $\E$ and $\ulrp(A^{\env})$ is $\otimes_A$. 

\begin{thm}
\label{thm:spce}
Let $\supp$ be the Balmer support for $(\umod(A^{\env}),\otimes_k,k)$, and consider $A$ as an object of $\umod(A^{\env})$ under the standard bimodule structure. Then 
\begin{enumerate}[\qquad \rm(a)]
\item $\Spc (\E) \cong \supp(A) \subseteq \Spc (\umod(A^{\env}))$, and
\item $\Spc (\E) \cong \Spc (\umod(A))$. 
\end{enumerate}
If $\Spc (\umod(A))$ is Noetherian, then the action of $\E$ on $\umod(A)$ is parametrizing.
\end{thm}

\begin{proof}
We apply Proposition \ref{prop:support-subcat-spc} to $\K_1 = \E$ and $\K_2 = \umod(A^{\env})$,
noting that $\E$ is a full triangulated subcategory of $\umod(A^{\env})$ and both $\umod(A^{\env})$ and $\E$ are unipotent.
This proves (a).  
For (b), first note that by Proposition \ref{prop:findimfunctor}(b) the action of $\E$ on $\umod(A)$ is conservative, 
and by Proposition~\ref{prop:Hopf-GF} it is right tensor-invertible. 
The statement (b) now follows from Corollary \ref{cor:spck-vs-spca-unip}. 
The last claim follows from Proposition \ref{prop:stev-supp}.
\end{proof}

Although we cannot give a description which is quite as clean for $\Spc (\ulrp(A^{\env}))$, nor for $\Spc(\K)$ where $\K \subseteq \ulrp(A^{\env})$
is an arbitrary monoidal triangulated subcategory,
we are at least able to conclude that points are all images of points from $\Spc(\umod(A))$.

\begin{thm}
\label{thm:spclrp}
Suppose $\Spc (\umod(A))$ is Noetherian, and let $\K \subseteq \ulrp(A^{\env})$ be a rigid monoidal triangulated subcategory. Then there is a surjective map $\xi:\Spc (\umod(A)) \to \Spc (\K)$. Furthermore:
\begin{enumerate}[\qquad \rm(a)]
\item If $\K$ has a thick generator, then $\xi$ is continuous, and consequently $\Spc(\K)$ is Noetherian, and the action of $\K$ on $\umod(A)$ is parametrizing.
\item If $H$ is a subgroup of $\Aut(A)$ such that $_{\phi}A_1 \in \K$ for all $\phi \in H$, then $\xi$ factors through the orbit space $\Spc(\umod(A))/H$.
\item If $H$ is a subgroup of $\Aut(A)$ such that $\K$ is precisely the monoidal triangulated subcategory 
\[
\Thick({_\phi}A_1 \mid \phi \in H),
\]
then the points of $\Spc(\K)$ are in bijection with $H$-orbits of points of $\Spc(\umod(A))$.
\item If $H$ and $\K$ are as in (c), and $H$ is a finite group, then $\Spc(\K)$ is precisely the topological quotient $\Spc(\umod(A)) / H$.
\end{enumerate} 
\end{thm}

\begin{proof}
The existence of a surjective map $\xi\colon \Spc (\umod(A)) \to \Spc (\K)$ follows from Proposition \ref{prop:map-on-spectra-t-k}, since by Proposition \ref{prop:findimfunctor}(b) the action of $\K$ on $\umod(A)$ is conservative, and by Proposition~\ref{prop:Hopf-GF} it is right tensor-invertible. 
It also follows that $\Spc (\K)$ is Noetherian and the action of $\K$ on $\umod(A)$ is parametrizing when $\K$ has a thick generator, since in that case $\xi$ is continuous and the Stevenson support is described by Proposition \ref{prop:stev-supp}.

To see that $\xi$ factors through the orbit space as in (b), suppose $\P\in\Spc (\umod (A))$ and 
$X \in \xi(\P)$, so that $G(X') \in \P$ for all $X' \in \langle X \rangle$. Then for any $\phi \in H$ and $X' \in \langle X \rangle$, $_{\phi}A_1 \otimes_A X' \otimes_A A \cong {_\phi}X' \otimes_A A \in \P,$ that is, $X' \otimes_A A  \in (\phi^{-1}).\P.$ It follows that $\xi$ is constant on the $H$-orbit of $\P$.  

Note that the subcategory described in (c) is indeed a rigid monoidal triangulated subcategory of $\ulrp(A^{\env})$, since it is thickly generated by a collection of objects which are closed under tensor product and closed under duality. In this case, we can check that if $\xi(\P)=\xi(\Q)$, then $\P$ and $\Q$ are in the same $H$-orbit using the formula for $\xi$. Namely, $X \in \xi(\P)=\xi(\Q)$ if and only if $G(X') \in \P$ for all $X' \in \langle X \rangle$, if and only if $G(X') \in \Q$ for all $X' \in \langle X \rangle$.
By Lemma \ref{lem:thick-ideal-tens},
since $\K$ has thick generators of the form $_{\phi} A_1$ for $\phi \in H$,
for all $X' \in \langle X \rangle$ there exist $\phi_1, \phi_2 \in H$ for which
\[
X' \in \Thick({_{\phi_1} A_1} \otimes_A X \otimes_A {_{\phi_2} A_1} ) = \Thick({_{\phi_1}X}_{\phi_2^{-1}}).
\]
Now, since $G(_{\phi}A_1) ={_{\phi}A}_1 \otimes_A k \cong k$, we have $G(X') \in \Thick(G({_{\phi_1}X}))= \Thick({{_{\phi_1}}G}(X)).$ In other words, $X \in \xi(\P)$ if and only if ${_{\phi}G(X)} \in \P$ for all $\phi \in H$, which must hold if and only if ${_{\phi}G(X)} \in \Q$ for all $\phi \in H$ by the assumption that $\xi(\P)=\xi(\Q)$. Since $G$ is essentially surjective, we can rewrite this condition as: $M \in \umod(A)$ satisfies $M \in \phi.\P$ for all $\phi \in H$ if and only if $M \in \phi.\Q$ for all $\phi \in H$. In other words, the $H$-core of $\P$ (as in \cite[Definition 7.1]{HuangVashaw2025}) is equal to the $H$-core of $\Q$. By \cite[Proposition 7.2]{HuangVashaw2025}, this means that $\P$ and $\Q$ are in the same $H$-orbit. Parts (c) and (d) follow immediately.
\end{proof}

\subsection{The left-right projective category for a cyclic $p$-group}
\label{subsect:cyclic}

We now give some additional attention to the special case of modular representation theory, starting with cyclic groups. Let $p$ be a prime, $k$ an algebraically closed field of characteristic $p$,
and $A = k \mathbb{Z}/p\mathbb{Z}$, the group algebra of $\mathbb{Z}/p\mathbb{Z}$. 
Then $A^{\env}$ is the group algebra of $\mathbb{Z}/p\mathbb{Z} \times \mathbb{Z}/p\mathbb{Z}$.
Note that $A$ is unipotent, so we can apply the results of the previous sections. 
We use the presentations of $A$ and $A^{\env}$:
\[
    A \cong k[w]/(w^p) \ \text{ and } \ 
   A^{\env} \cong k[ u,v] / (u^p, v^p)
\]
where $u \in A^{\env}=A \otimes_k A^{\opp}$ corresponds to $w \otimes_k 1$ and $v$ corresponds to $1 \otimes_k w$. 

The Balmer spectrum of $\umod(A)$ is a single point (see \cite{Bal}), and therefore by Theorem \ref{thm:spclrp}, $\Spc (\ulrp(A^{\env}))$ is also a single point. In this section, as a proof-of-concept, we show how this fact can also be deduced using classical results of rank varieties, which we recall now. 

By results of Benson--Carlson--Rickard \cite{BCR}, thick subcategories of $\umod(A^{\env})$ are in bijection with unions of closed subsets of $\mathbb{P}^1_k$ (here and below, when we write $\mathbb{P}^n_k$ we mean under the Zariski topology), and consequently the Balmer spectrum of $\umod(A^{\env})$ is homeomorphic to $\mathbb{P}^1_k$. Under this identification, the Balmer support of $M \in \umod(A^{\env})$ corresponds to the rank variety of $M$ (as defined by Carlson in \cite{Carlson1983}), which corresponds to the cohomological support variety in a natural way by the Avrunin--Scott Theorem \cite{AvruninScott1982}.

Explicitly, for any $\alpha \in k$ the closed point $[\alpha : 1] \in \mathbb{P}^1_k$ is in the rank variety of an $A^{\env}$-module $M$ if $M$ restricted to the subalgebra generated by $\alpha u + v + 1$ in $A^{\env}$ is not free, and contains the point $[1:0]$ if $M$ restricted to the subalgebra generated by $u+1$ is not free.

From this description, we can see that under the bijection between unions of closed subsets of $\mathbb{P}^1_k$ and thick subcategories of $\umod(A^{\env})$, the thick subcategory $\ulrp(A^{\env}) \subseteq \umod(A^{\env})$ corresponds to all points $q$ of $\mathbb{P}^1_k$ such that $[1:0]$ and $[0:1]$ are not in the closure of $q$. That is, the support of $\ulrp(A^{\env})$ is the maximal specialization-closed subset of $\mathbb{P}^1_k$ that does not contain $[1:0]$ and $[0:1]$, or, equivalently, $\ulrp(A^{\env})$ consists of all $M \in \umod(A^{\env})$ such that $[1:0]$ and $[0:1]$ are not in the support of $M$. Note that since this is not a closed subset, $\ulrp(A^{\env})$ cannot have a single thick generator.

Given $\gamma \in k^{\times}$, there is a corresponding automorphism of $A\cong k[w]/(w^p)$ that rescales $w$ by $\gamma$. By abuse of notation, we denote this automorphism again by $\gamma$. 

\begin{prop}
  \label{prop:klein-nonzero-inv}
  Let $A = k \mathbb{Z}/p\mathbb{Z}$.
Every nonzero thick subcategory of $\ulrp(A^{\env})$ contains an object of the form $_{\gamma}A_1$, for some $\gamma \in k^{\times}$. 
\end{prop}

\begin{proof}
Let $\I$ be a thick subcategory in $\ulrp(A^{\env})$. 
Then $\I$ is also a thick subcategory in $\umod(A^{\env})$. 
Fix an object $M \in \I$. 
The rank variety of $M$ must contain some closed point which is necessarily not $[1:0]$ or $[0:1]$, say $[\alpha:1]$, where $\alpha \not = 0$. 

Set $\gamma= - 1/\alpha$. We claim that the variety of $_\gamma A_1$ is precisely this point, from which it will follow that $_\gamma A_1$ is in the thick subcategory generated by $M$. Note that $_\gamma A_1$ is $p$-dimensional with basis $1,  w, w^2,\ldots , w^{p-1}$.
The element $u$ of $A^{\env}$ acts by sending $w^i$ to $\gamma w^{i+1}$ for all $i$, and the element $v$ acts by sending $w^i$ to $w^{i+1}.$ Indeed, if we restrict $_\gamma A_1$ to the subalgebra generated by $\beta u+v +1,$ for $\beta \not = \alpha$, then we see that the element $\beta u +v$ sends $w^i$ to a nonzero scalar multiple of $w^{i+1}$ (for $i=0,1,\ldots, p-2$), hence the module is free. 
However, for $\beta =\alpha,$ the element $\beta u+v$ acts by 0, hence the module is not free. 

Since $[\alpha:1]$ is the only closed point in the variety of $_\gamma A_1$, the variety consists precisely of that point. It follows that $_\gamma A_1$ is in the thick subcategory generated by $M$, and consequently in $\I$. 
\end{proof}

\begin{cor}
  Let $A = k\mathbb{Z}/p\mathbb{Z}$. 
    Considered as a \mtc\ under $\otimes_A$, the only thick ideal of $\ulrp(A^{\env})$ is the zero ideal. As a consequence, $\Spc (\ulrp(A^{\env}))$ is the one-point topological space.
\end{cor}

\begin{proof}
By Proposition \ref{prop:klein-nonzero-inv},
every nonzero thick subcategory of $\ulrp(A^{\env})$ contains some $_\gamma A_1$.
Since $_\gamma A_1$ is tensor-invertible under $\otimes_A$ by Proposition \ref{prop:Bass}, there are no proper nonzero thick ideals of $\ulrp(A^{\env})$. 
\end{proof}

\subsection{The left-right projective category for an elementary abelian $p$-group}

Now let $A$ be the group algebra of $(\mathbb{Z}/p \mathbb{Z})^{ n}$, where again we take $k$ to be an algebraically closed field of characteristic $p$. The Balmer spectrum of $\umod(A^{\env})$ is homeomorphic to $\mathbb{P}^{2n-1}_k$, again in a way which identifies the Balmer support with the rank variety support. A similar analysis to that of the previous section can be used to show that the rank variety of $\ulrp(A^{\env})$ is the maximal specialization-closed subset of $\mathbb{P}^{2n-1}_k$ that does not contain any point $[\alpha_1:\ldots:\alpha_{2n}]$ with $\alpha_1=\alpha_2\ldots= \alpha_n =0$ or $\alpha_{n+1}=\alpha_{n+2} = \ldots=\alpha_{2n}=0$. Presenting $A$ as
\[
A \cong k[w_1,\ldots,w_n]/(w_1^p,\ldots,w_n^p),
\]
given any $\phi \in \GL_n(k)$ we have a corresponding automorphism of $A$ given by applying $\phi$ to the $k$-span of $w_1,\ldots, w_n$, a vector space isomorphic to $k^n$; we will also denote this automorphism by $\phi$. One can check that the rank variety of $_{\phi} A_1$ is the closed set whose closed points are all $[\alpha_1:\ldots:\alpha_{2n}]$ for which
\[
\phi \left ( \begin{bmatrix} \alpha_1 \\\ldots\\ \alpha_n \end{bmatrix} \right ) = -\begin{bmatrix} \alpha_{n+1} \\\ldots \\ \alpha_{2n} \end{bmatrix}. 
\]
In particular, the support of each $_{\phi}A_1$ is isomorphic to $\mathbb{P}^{n-1}_k \cong \Spc (\umod(A))$, which illustrates the general result of Theorem \ref{thm:spce} 
that $\supp(A) \subseteq \Spc (\umod(A^{\env}))$ is homeomorphic to $\Spc (\umod(A))$.

Note that we can no longer conclude, as we could in the previous section, that every thick subcategory of $\ulrp(A^{\env})$ contains an invertible bimodule of the form $_{\phi} A_1$. Indeed, there exist thick subcategories of $\ulrp(A^{\env})$, for instance those corresponding to single closed points, which do not contain any $_{\phi} A_1$. 

Nevertheless, we can still deduce from the general results of the previous section that $\Spc (\ulrp(A^{\env}))$ has a unique closed point, in other words, the zero ideal of $\ulrp(A^{\env})$ is prime. In the language of \cite[Definition 4.1]{Balmer2010}, where categories satisfying this property were studied, this would mean that $\ulrp(A^{\env})$ is {\emph{local}}.

\begin{prop}
  \label{prop:ulrp-elemab}
  Let $A =  k (\mathbb{Z}/p\mathbb{Z})^n$.
The Balmer spectrum of $\ulrp(A^{\env})$ has a unique closed point.
\end{prop}

\begin{proof}
  By Theorem \ref{thm:spclrp}, there is a surjective map $\mathbb{P}^{n-1}_k \to \Spc (\ulrp(A^{\env}))$ which factors through $\mathbb{P}^{n-1}_k/\Aut(A)$. The action of $\Aut(A)$ on closed points of $\mathbb{P}^{n-1}_k$ is transitive: if $M$ is a module with rank support consisting of the single closed point $[\alpha_1 :  \ldots: \alpha_n ]$ then for any $\phi \in \GL_n(k)$, 
the support of ${_\phi} M$ is $\phi^{-1}([\alpha_1 : \ldots : \alpha_n])$.
 Therefore, the map $\mathbb{P}^{n-1}_k \to \Spc(\ulrp(A^{\env}))$ is constant on closed points of $\mathbb{P}^{n-1}_k$, that is, it must send all closed points of $\mathbb{P}^{n-1}_k$ to a single closed point $q$ of $\Spc (\ulrp(A^{\env}))$. 
\end{proof}

\begin{exam}
By Theorem \ref{thm:spclrp}, there is a monoidal triangulated subcategory of $\ulrp(A^{\env})$ whose spectrum has points that are in bijection with projective equivalence classes of closed subsets of $\mathbb{P}^n_k$. Recall that two closed subsets of projective space are called projectively-equivalent if there is some $\phi$ in $\GL_n(k)$ transforming one closed set into the other. Recall also that projective equivalence is a stronger concept than isomorphism, and that the classification of projective equivalence classes is a difficult problem in geometric invariant theory. 
\end{exam}

We illustrate $\ulrp(A^{\env})$, $\E$, and Proposition \ref{prop:ulrp-elemab} using the following picture. The Balmer spectrum of $\umod(A^{\env})$ is homeomorphic to $\mathbb{P}^{2n-1}$, and for objects of $\ulrp(A^{\env})$ we can consider their support as objects of $\umod(A^{\env})$. In this picture, we use the notation $[\vec{\alpha}:\vec{\beta}]$ for points in $\mathbb{P}^{2n-1}$ to represent $[\alpha_1: \alpha_2:...:\alpha_n:\beta_1:\beta_2:...\beta_n]$, for $\vec{\alpha}, \vec{\beta} \in k^n$. We have:

\begin{tikzpicture}[scale=1]

  \draw[thick]
    plot[smooth cycle, tension=1]
      coordinates{(0,0) (4,0.4) (4.2,2.7) (2,4) (-0.3,2.5)};

  \node at (2,4.5)
    {$\Spc(\umod(A^{\env})) \cong \mathbb{P}^{2n-1}$};

  \coordinate (Supp) at (1.2,2.6);
  \coordinate (Wone) at (1.0,1.0);
  \coordinate (Wtwo) at (3.1,1.4);

  \def\R{0.7}
  \def\equator{0.25}

  \draw[thick] (Supp) circle (\R);
  \draw (Supp) ++(-\R,0)
     arc[start angle=180, end angle=360, x radius=\R, y radius=\equator];
  \draw[dashed] (Supp) ++(\R,0)
     arc[start angle=0, end angle=180, x radius=\R, y radius=\equator];

  \node[right] (SuppLabel) at (4.7,3.0)
    {\shortstack[l]{$\supp(A) \cong \mathbb{P}^{n-1} \cong \Spc(\E)$\\[2pt]
    $=\,\{[\vec{\alpha}:-\vec{\alpha}] \mid \vec{\alpha} \in k^n\setminus \{\vec{0}\}\}$}};
  \draw (Supp) -- (SuppLabel.west);

  \fill[red!10] (Wone) circle (\R);
  \draw[thick,red] (Wone) circle (\R);
  \draw[red] (Wone) ++(-\R,0)
     arc[start angle=180, end angle=360, x radius=\R, y radius=\equator];
  \draw[red,dashed] (Wone) ++(\R,0)
     arc[start angle=0, end angle=180, x radius=\R, y radius=\equator];

  \node[left] (WoneLabel) at (-1.2,1.0)
    {\shortstack[l]{$W_1 \cong \mathbb{P}^{n-1}$\\[2pt]
    $=\,\{[\vec{\alpha}:\vec{0}] \mid \vec{\alpha} \in k^n\setminus\{\vec{0}\}\}$}};
  \draw (WoneLabel.east) -- (Wone);

  \fill[red!10] (Wtwo) circle (\R);
  \draw[thick,red] (Wtwo) circle (\R);
  \draw[red] (Wtwo) ++(-\R,0)
     arc[start angle=180, end angle=360, x radius=\R, y radius=\equator];
  \draw[red,dashed] (Wtwo) ++(\R,0)
     arc[start angle=0, end angle=180, x radius=\R, y radius=\equator];

  \node[right] (WtwoLabel) at (4.7,1.2)
    {\shortstack[l]{$W_2 \cong \mathbb{P}^{n-1}$\\[2pt]
    $=\,\{[\vec{0}:\vec{\beta}] \mid \vec{\beta} \in k^n\setminus\{\vec{0}\}\}$}};
  \draw (Wtwo) -- (WtwoLabel.west);

\end{tikzpicture}

The category $\ulrp(A^{\env})$ consists of all objects whose $\umod(A^{\env})$-support does not include any points of $W_1$ or $W_2$, and the category $\E$ consists of all objects whose $\umod(A^{\env})$-support is contained in $\supp(A)$. Proposition \ref{prop:ulrp-elemab} is a reflection of the fact that if $\I$ is any thick ideal of $\ulrp(A^{\env})$, then the $\umod(A^{\env})$-support of $\I$ contains all closed points of $\supp(A)$.

\subsection{The left-right projective category for a general finite $p$-group}

Let $G$ be a finite $p$-group and
$k$ an algebraically closed field of characteristic $p$. 
The group algebra $A = kG$ is again a unipotent Hopf algebra. 
For brevity, denote by $V_G^{\maxx}$ the maximal ideal spectrum $\MaxSpec (\cohom(G,k))$ of the cohomology ring of $G$. 
For every elementary abelian $p$-subgroup $E$ of $G$, the restriction functor $\res_E^G\colon \smod(G) \to \smod(E)$ induces a continuous map on the spectra of the respective cohomology rings, $f_E^G\colon V_E^{\maxx} \to V_G^{\maxx}$. Letting $M$ be a $kG$-module, denote by $V_G^{\maxx}(M)\subseteq V_G^{\maxx}$ the support variety of $M$ as in Section~\ref{sect:cohomsupport} based in $V_G^{\maxx}$, and similarly for $kE$-modules. Recall that
\begin{equation}
\label{eqn:var-rest}
V_E^{\maxx}({\res}_E^G(M)) = (f_E^G)^{-1}(V_G^{\maxx}(M))
\end{equation}
by \cite[Theorem 3.1]{AvruninScott1982}. Recall also that by (a naive version of) Quillen Stratification (see e.g.~ \cite[Proposition 5.6.1]{Benson1998}), $V_G^{\maxx}$ is the union of the subsets $f_E^G(V_E^{\maxx})$ over all elementary abelian subgroups $E$ of $G$. 

We will use our work above on elementary abelian $p$-groups 
to show that $\Spc (\ulrp(A^{\env}))$ is again local. 

\begin{thm}
\label{thm:spc-arbitrary-pgroup}
Let $A=kG$ for a finite $p$-group $G$.
The Balmer spectrum of $\ulrp(A^{\env})$ has a unique closed point.
\end{thm}

\begin{proof}
By Theorem~\ref{thm:spclrp}, 
it suffices to show that the submodule category generated by any object $M \in \umod(A)$ has every closed point of $\Spc(\umod(A))$ in its support. 
The support variety of a nonzero submodule category must contain some closed point, say $\mathfrak{m}$, understood as a point of $V_G^{\maxx}$. Then there exists a $kG$-module $M$ such that $V_G^{\maxx}(M)$ is the smallest closed projective subspace containing $\mathfrak{m}$. 
By the above discussion, there is then some elementary abelian subgroup $E$ of $G$ such that $\mathfrak{m}$ is in $f_E^G(V_E^{\maxx})$.
In particular, $\res^G_E(M)$ is not projective. 

By Proposition \ref{prop:ulrp-elemab}, since $\res_E^G(M)$ is not projective, the $\ulrp(kE^{\env})$-submodule category of $\umod(kE)$ generated by $\res_E^G(M)$ has all closed points of $\Spc(\umod(kE))$ in its support. 
In particular, given any other closed point $\mathfrak{n}$,  
there is some $kE$-bimodule $B$ in $\ulrp(kE^{\env})$ such that $k$ is in the thick subcategory generated by $B \otimes_{kE} \res_E^G(M)$. 

Now note that $B \otimes_{k E} \res_E^G(M)$ is a summand of 
\begin{align*}
{\res}_E^G ({\ind}_{E \times E^{{\opp}}}^{G \times G^{{\opp}}}(B) \otimes_{kG} M) &\cong {\res}_E^G( (kG \otimes_{kE} B \otimes_{kE} kG) \otimes_{kG} M)\\
& \cong {\res}_E^G({\ind}_E^G(B \otimes_{kE} {\res}_E^G(M))).
\end{align*}
Since the support variety of $B \otimes_{kE} \res_E^G(M)$ contains $\mathfrak{n}$, it follows from (\ref{eqn:var-rest}) that the support variety of ${\ind}_{E \times E^{{\opp}}}^{G \times G^{{\opp}}}(B) \otimes_{kG} M$ contains the corresponding point of $f_E^G(V_E^{\maxx})$. 
Since $B$ is in $\ulrp (kE^{\env})$, the induced module $\ind_{E\times E^{{\opp}}}^{G\times G^{{\opp}}} (B)$
is in $\ulrp (kG^{\env})$. 
Therefore ${\ind}_{E \times E^{{\opp}}}^{G \times G^{{\opp}}}(B) \otimes_{kG} M$ is in any $\ulrp(A^{\env})$-submodule category of $\umod(A)$ which contains the module $M$.

To summarize, we have shown that if an $\ulrp(kG^{\env})$-submodule category of $\umod(kG)$ has a closed point of $f_E^G(V_E^{\maxx})$ in its support variety for some elementary abelian subgroup $E$, then that submodule category must have all closed points of $f_E^G(V_E^{\maxx})$ in its variety. The topological space $V_G^{\maxx}$ is connected~\cite{Carlson1984}, and is the direct limit of the $f_E^G(V_E^{\maxx})$-subspaces by \cite[Corollary 5.6.4]{Benson1998}, so the support variety of our submodule category is either empty or all of $V_G^{\maxx}$. 
\end{proof}

 \begin{exam}
We note that in extreme cases, this unique closed point is the entire space $\Spc(\ulrp(A^{\env}))$.
Namely, if $G$ is a cyclic $p$-group, or if $p=2$ and $G$ is a generalized quaternion group, then $G$ has a unique elementary abelian subgroup which is cyclic \cite[Theorem 5.4.10(ii)]{Gorenstein1980}, and $\Spc(\umod(A))$ is a single point, so that $\Spc(\umod(A)) \cong \Spc(\ulrp(A^{\env}))$. 
\end{exam}

\subsection{Hochschild cohomology and support for $\E$}
 We return to the more general case of a finite dimensional unipotent Hopf algebra $A$. It remains to discuss the connections among Balmer supports for $\E$ and $\ulrp(A^{\env})$, Stevenson supports for $\umod(A)$, and the classically-defined Hochschild support. We first discuss these connections for the category $\E$, and in the subsequent section discuss the connections for $\ulrp(A^{\env})$.

We say that the Hochschild cohomology ring $\HH(A)$ of $A$ 
satisfies the {\emph{finite generation condition (fg)}} if 
\begin{enumerate}
\item $\HH(A)$ is a finitely generated algebra over k, and 
\item $\Ext^{\bu}_{A^{\env}}(B,B)$ is a finitely generated $\HH(A)$-module, for all $B \in \umod(A^{\env})$.
\end{enumerate}
By \cite[Proposition 1.4]{EHSST2004}, the second condition is equivalent to requiring finite generation of the modules $\Ext^{\bu}_A(M,N)$ over $\HH(A)$, for all $M, N \in \umod(A)$. 

Recall that since $A^{\env}$ is unipotent, we have comparison maps
\begin{align*}
\rho_{A^{\env}}\colon \Spc(\umod(A^{\env})) &\to \Spech (\cohom(A^{\env})),\\
\rho_{\E}\colon \Spc(\E) & \to \Spech (\HH(A)),
\end{align*}
see Section \ref{subsect:catcenter}. Recall also that it is conjectured that $\rho_{A^{\env}}$ takes values in $\Proj(\cohom(A^{\env}))$, and gives a homeomorphism $\Spc(\umod(A^{\env})) \cong \Proj (\cohom(A^{\env}))$ \cite[Conjecture E]{NVY3}.

\begin{thm}
  \label{thm:nvy-conj-gives-hochschild}
  Let $A$ be a finite dimensional unipotent Hopf algebra.
    Suppose that $(\umod(A^{\env}), \otimes_{k})$ satisfies \cite[Conjecture E]{NVY3}, and $\HH(A)$ satisfies (fg). Then $\Spc (\E)$ is homeomorphic to $\Proj (\HH(A))$ in a way which identifies the Balmer support with the Hochschild cohomology support.
\end{thm}

\begin{proof} 
The map $\rho_{\E}\colon \Spc (\E) \to \Spech \HH(A)$ actually takes values in $\Proj( \HH(A))$, by an argument analogous to \cite[Corollary 7.1.3]{NVY3}, since by \cite[Proposition 8.3.6]{Witherspoon2019}
the complexity of a bimodule is equal to the dimension of its support variety. 
By \cite[Theorem 7.2.1(b)]{NVY3}, this map $\rho_{\E}\colon \Spc (\E) \to \Proj( \HH(A))$ is surjective. 

The map $\rho_{\E}\colon \Spc (\E) \to \Proj (\HH(A))$ is also injective. To prove this, note that 
by Proposition \ref{prop:rho-diag}, there is a commutative diagram of the form
    \begin{center}
        \begin{tikzcd}
\Spc (\E) \arrow[d, hook] \arrow[r, two heads] & \Proj(\HH(A)) \arrow[r]            & \Spech(\HH(A)) \arrow[d] \\
\Spc (\umod(A^{\env})) \arrow[r, "\cong"]       & \Proj(\cohom(A^{\env})) \arrow[r, hook] & \Spech(\cohom(A^{\env}))     
\end{tikzcd}
    \end{center}
Injectivity of the map $\Spc (\E) \to \Proj (\HH(A))$ now follows from the injectivity of the composition defined by first going down (see Proposition \ref{thm:spce}), and then going right along the diagram.

To prove that the map $\rho_{\E}\colon \Spc (\E) \to \Proj(\HH(A))$ is a homeomorphism, we must show that $\rho_{\E}$ is closed. This will follow by the arguments of \cite[Proposition 8.4]{CV} (although that proposition is stated for finite tensor categories, the arguments only use surjectivity of $\rho_{\E}$ and finite generation properties of cohomology). 

Note that the identification of supports also follows by the arguments from \cite[Proposition 8.4]{CV}. 
\end{proof}

The following corollary is interesting to note, since the extent to which the tensor product property (v$'$) of Remark \ref{rem:tpp} holds for support varieties is currently an open question which has generated significant interest. Recall that any support which satisfies the tensor product property also satisfies the noncommutative tensor product property of Definition~\ref{defn:support-datum}(v).

\begin{cor}
    Assume the hypotheses of Theorem \ref{thm:nvy-conj-gives-hochschild}. Then Hochschild cohomology support, for objects of $\E$, satisfies the tensor product property {\rm (v$'$)} of Remark~\ref{rem:tpp}.

\end{cor}

\begin{proof}
    The Balmer support for $\E$ satisfies the tensor product property by Remark~\ref{rem:tpp} since $\E$ is generated by its tensor unit.
This support is identified with the Hochschild cohomology support on $\E$ by Theorem \ref{thm:nvy-conj-gives-hochschild}.
\end{proof}

Note that there cannot be a general theorem stating that the Hochschild cohomology support for an arbitrary finite dimensional algebra satisfies the tensor product property for all bimodules; see \cite{BHS2020}. 
Indeed, by the results of \cite{BHS2020}, Hochschild cohomology support for general bimodules can satisfy neither the commutative nor the more general noncommutative tensor product properties. The expectation we have in view of \cite{NVY3} is that under favorable conditions, support varieties based on a suitable categorical center of the Hochschild cohomology ring for $\ulrp(A^{\env})$ will satisfy the noncommutative, but not necessarily commutative, tensor product property.

\subsection{Categorical center of Hochschild cohomology and support for $\ulrp$}
In light of our results above for some unipotent Hopf algebras $A$,
here we speculate what happens for more general algebras $\Lambda$
by collecting some partial results about categorical centers (recall Section \ref{subsect:catcenter}) and Picard groups
on the one hand, and Drinfeld centers of bimodule categories on the other.
These partial results indicate that we may not expect 
cohomologically-defined supports or Balmer supports to be large enough for utility. 
Instead, our reliance on well-chosen subcategories, such as $\E$,  
may be necessary for a good theory. For elementary abelian $p$-groups, we show that the categorical center of the Hochschild cohomology ring is just the field $k$, so that by Theorem \ref{thm:spc-arbitrary-pgroup} the comparison map between $\Spc (\ulrp(A^{\env}))$ and $\Proj (\ccent)$ is a homeomorphism. 

We first discuss the categorical center in the more general setting of finite dimensional selfinjective algebras $\Lambda$. Note that each algebra automorphism $\psi$ of $\Lambda$ induces an algebra automorphism of 
its enveloping algebra $\Lambda^{\env}$
by acting on each tensor factor.
By Proposition~\ref{prop:Bass}(a), 
$\psi$ may also be viewed as an $\Lambda^{\env}$-module
isomorphism $\Lambda\stackrel{\sim}{\longrightarrow}  {}_{\psi}\Lambda_{\psi}$. 
Given an injective $\Lambda^{\env}$-module resolution $I$ of $\Lambda$, 
by the Comparison Theorem, 
this $\Lambda^{\env}$-module isomorphism $\psi\colon \Lambda\stackrel{\sim}{\longrightarrow}
{}_{\psi}\Lambda_{\psi}$
may be lifted to a chain map $\psi\colon I\rightarrow {}_{\psi}I_{\psi}$. 
In this way we obtain an action of $\psi$ on $\Sigma^i \Lambda$ for each $i$.

\begin{lem}
Let $\psi$ be an algebra automorphism of $\Lambda$ and let 
$ g\in\Hom_{\Lambda^{\env}}(\Lambda, \Sigma^i \Lambda)$. 
If $g$ is in the categorical center
of $\HH(\Lambda)$, with respect to a generating set of $\ulrp(\Lambda^{\env})$
that includes the $\Lambda$-bimodule ${}_{\psi}\Lambda_1$, 
then $ \psi ( g) = g$. 
\end{lem}

\begin{proof}
  Set $M = \psiL$, an invertible $\Lambda$-bimodule with inverse $M\ldual \cong \Lpsi$
  by Proposition~\ref{prop:Bass} and the definition of left dual (see Sections \ref{subsect:duals}
and \ref{subsect:dual-selfinj}).
Note that for $\Lambda$-bimodules $N$, there are functorial $\Lambda$-bimodule isomorphisms
\[
M\ot_{\Lambda} N \ot_{\Lambda} M\ldual\stackrel{\sim}{\longrightarrow} {}_{\psi} N _{\psi} .
\]
Denote by $\sigma_M (g)$ the composition 
(as in~\cite[(4.1)]{NVY3}):
\[
\xymatrix{
  \Lambda\ar[r]^{\coev \hspace{1.9cm}} & M\ot_{\Lambda} M\ldual\stackrel{\sim}{\longrightarrow}
  M\ot_{\Lambda} \Lambda \ot_{\Lambda} M\ldual
\ar[r]^{\hspace{.7cm}\id_M\ot g \ot\id_{M\ldual}} & M\ot_{\Lambda} (\Sigma^i\Lambda)\ot_{\Lambda} M\ldual \\
& \stackrel{\sim}{\longrightarrow}\Sigma^i(M\ot _{\Lambda} \Lambda \ot_{\Lambda} M\ldual)
  \stackrel{\sim}{\longrightarrow} \Sigma^i (M\ot_{\Lambda} M\ldual) 
   \ar[r]^{\hspace{2.9cm}\Sigma^i \coev^{-1}} & \Sigma^i \Lambda . 
}
\]
By the above observations, the $\Lambda$-bimodule 
$M\ot_{\Lambda} \Lambda\ot_{\Lambda} M\ldual$ is isomorphic to 
${}_{\psi}\Lambda_{\psi}$ under the isomorphism $\Lambda \rightarrow {}_{\psi}\Lambda_{\psi}$
given by $a\mapsto \psi(a)$ for all $a\in \Lambda$.
Since this isomorphism is functorial, given 
$g\colon \Lambda\rightarrow \Sigma^i\Lambda$, there is a corresponding map 
${}_{\psi}g_{\psi} \colon {}_{\psi} \Lambda _{\psi} 
\rightarrow {}_{\psi} (\Sigma^i \Lambda)_{\psi}$.
The composition in the above sequence of maps that gives 
\[
  M\ot_{\Lambda} (\Sigma^i \Lambda)\ot_{\Lambda} M\ldual\rightarrow \Sigma^i \Lambda
\] 
corresponds to the isomorphism ${}_{\psi}(\Sigma^i \Lambda)_{\psi}\rightarrow \Sigma^i\Lambda$
given by applying $\psi^{-1}$, arising from the
inverse of the chain map $\psi\colon I\rightarrow {}_{\psi}I_{\psi}$. 
Under the above composition therefore, $\sigma_M(g)
= \psi^{-1} (g)$ defined by
$(\psi^{-1}(g))(a) = \psi^{-1}(g (\psi(a)))$ for all $a\in \Lambda$,
the standard action of a group on functions.
The result now follows by~\cite[Lemma 4.2.1(c)]{NVY3}. 
\end{proof}

The above lemma is quite limiting for algebras that have many automorphisms. 
Further limitations are in~\cite[Section 3.3]{NVY3}, 
where it is stated that when taking all objects of $\K = \ulrp(\Lambda^{\env})$
as a generating set, the categorical center 
$\ccent$ of $\HH(A)$ is isomorphic to the cohomology
ring of the Drinfeld center of $\K$.
Turning from these stable categories back to the categories of bimodules themselves,
we recall the following related proposition.
It is a special case of~\cite[Theorem 3.3]{Schauenburg}, taking 
both underlying categories there to be the category $\VVec_k$ of vector spaces over $k$.
Alternatively, it is a special case of~\cite[Theorem 2.10]{ACM2012},
taking $k$ there to be a field.

\begin{prop}[see \cite{ACM2012,Schauenburg}]\label{prop:ctr-mon-cat}
  The center of the monoidal category $\smod (A^{\env})$ of $A$-bimodules
  is equivalent to $\VVec_k$.
  \end{prop}

We therefore pose a natural question:
\begin{quest}
What is the Drinfeld center of the stable monoidal triangulated category $\ulrp (\Lambda^{\env})$
for a finite dimensional selfinjective algebra $\Lambda$, 
or more specifically for a finite dimensional unipotent Hopf algebra?
\end{quest}

Let $A$ be the group algebra 
of an elementary abelian $p$-group and $\K = \ulrp(A^{\env})$,
with chosen generating set including all 
invertible bimodules ${}_{\alpha}A_{\beta}$. 

\begin{prop}
\label{prop:elemab-catcent}
  Let $G$ be an elementary abelian $p$-group, $k$ an infinite field
  of characteristic $p$, and $A=kG$.
  The categorical center of $\HH(A)$,
  with respect to a generating set of $\ulrp(A^{\env})$ that includes
  all ${}_{\alpha}A_{\beta}$, is
  isomorphic to $k$ (concentrated in degree 0).
  \end{prop}

\begin{proof}
  First note that $A\cong k[x_1,\ldots,x_n]/(x_1^p,\ldots, x^p_n)$ for some $n$.
  The group $\Aut_k(A)$ of algebra automorphisms of $A$ includes
  as a subgroup a copy of 
  $\GL_n(k)$ acting on the vector space $V$ with basis $x_1,\ldots,x_n$
  and action canonically extended to an algebra automorphism of $A$.

  By~\cite{Holm} (see also~\cite[Corollary 9.5.5]{Witherspoon2019}),
  there is an isomorphism of graded algebras,
\begin{equation}\label{eqn:abel-gp}
    \HH(A)\cong A \ot \cohom(A,k) ,
\end{equation}
    where $\cohom(A,k)$ is the group cohomology, also written $\cohom(G,k)$.
    Furthermore, by~\cite{SiegelWitherspoon} (see also~\cite[Corollary 9.4.7]{Witherspoon2019}), we have the following decomposition of this Hochschild cohomology ring.
Let $I = \ker (\varepsilon)$, the augmentation ideal of $A$, and $I^{\mathsf{ad}}$ the $A$-module structure on $I$ where the action is given by the adjoint action
$
a.m:=a_1 m S(a_2)
$ 
for $a \in A$ and $m \in I$.
Then
\begin{equation}\label{eqn:HH-ds}
\cohom(A,A)\cong \cohom(A, A^{\mathsf{ad}})\cong  \cohom(A,k)\oplus
\cohom(A, I^{\mathsf{ad}}) ,
\end{equation}
a direct sum of the subalgebra $\cohom(A,k)$ and ideal
$\cohom(A, I^{\mathsf{ad}})$ of $\HH(A,A)$.
    The subalgebra
    $\cohom(A,k)$ of $\HH(A)$ in the
    decomposition~(\ref{eqn:HH-ds}) can be identified with the subalgebra
    $k \ot \cohom(A,k)$ under the above isomorphism~(\ref{eqn:abel-gp}).
    Any algebra isomorphism $\psi$ of $A$ acts canonically on
     $\cohom(A,k)$, and its action on
    Hochschild cohomology  $\HH(A)$
    is factorwise in~(\ref{eqn:abel-gp}).

    Note that the invariant subring of $A$ under the action of $\GL_n(k)$
    is $k$.
    The group cohomology $\cohom(G,k)$ depends on $p$:
    \[
    \cohom(G,k) \cong \left\{\begin{array}{ll}
    k[y_1,\ldots, y_n] , \mbox{ with } |y_i|=1, & \mbox{ if }p=2, \\
    \bigwedge (y_1,\ldots,y_n)\ot k[z_1,\ldots, z_n] ,
    \mbox{ with } |y_i|=1, \ |z_i|=2, & \mbox{ if } p >2 .
    \end{array}\right.
    \]
    It will suffice to determine the action on $\cohom(G,k)$ 
    of the subgroup $\Gamma$ of $\GL_n(k)$ 
    generated by permutation matrices together
    with diagonal matrices diag$(c,1,\ldots,1)$ for $c\in k^{\times}$.

    To find the action of $\Gamma$ on  $\cohom(G,k)$,
    we first recall a calculation of this group cohomology $\cohom(G,k)$:
    Let $P_{\bu}$ be the total complex
    of the tensor product of the standard complexes for each
    tensor factor $A_i = k[x_i]/(x_i^p)$ of $A$:
\begin{equation}
\begin{xy}
  \xymatrix{
    \cdots\ar[r]^{x_i^{p-1}\cdot} & A_i\ar[r]^{x_i\cdot} & A_i
       \ar[r]^{x_i^{p-1}\cdot}& A_i \ar[r]^{x_i\cdot} &
       A_i 
       \ar[r] & 0
}
\end{xy}
\end{equation}
As a graded vector space,
the total complex $P_{\bu}$ is isomorphic to a polynomial ring in $n$ indeterminates;
let $\epsilon_{j_1,\ldots,j_n}$ denote a free
basis element $1\ot\cdots\ot 1$ of the free $A$-module $A\cong A_1\ot\cdots\ot A_n$
in degree $j_1+\cdots + j_n$ whose $i$th factor has degree $j_i$.
Apply $\Hom_A( - , k)$, noting that all differentials
will be 0. As a graded vector space then,
\[
\cohom(G,k) \cong \Hom_A(P_{\bu}, k) \cong\Hom_k(\widetilde{P}_{\bu},k)
\]
where $\widetilde{P}_m$ is the vector space having basis all
$\epsilon_{j_1,\ldots,j_n}$ with $j_1+\cdots + j_n=m$. 
In the above expression for group cohomology $\cohom(G,k)$,
it can be checked that  $y_i$ (respectively $z_i$)
is dual to $\epsilon_{0,\ldots, 0 , 1 , 0 , \ldots , 0}$
(respectively  $\epsilon_{0,\ldots, 0 , 2 , 0 , \ldots , 0}$)
where the nonzero entry is in the $i$th component.

Now if $\psi$ is a permutation matrix, given the above details,
we see that
$\psi$ induces the corresponding permutation of the
generators $y_1,\ldots,y_n$ and of $z_1,\ldots,z_n$.
If $\psi = \ {\mbox{diag}}(c,1,\ldots,1)$ for 
$c\in k^{\times}$,
it can be checked that
\[
\psi(\epsilon_{j_1,\ldots,j_n}) = \left\{
\begin{array}{ll}  c^{p \frac{j_1}{2}} \epsilon_{j_1,\ldots,j_n} , &
  \mbox{ if $j_1$ is even} , \\
      c^{p \frac{j_1-1}{2} + 1} \epsilon_{j_1,\ldots,j_n} , &
      \mbox{ if $j_1$ is odd} . 
\end{array}\right.
\]
Thus the action of $\psi$ on $y_1$ and on $z_1$ is multiplication
by specific powers of $c$.
Composing with a permutation matrix gives similar actions on $y_2,\ldots,y_n,
z_2,\ldots,z_n$. 
Passing to cohomology, since $k$ is infinite, 
none of $y_1,\ldots, y_n, z_1,\ldots,z_n$
is invariant under these automorphisms, 
nor is any positive degree element of $\cohom(G,k)$.
Thus the categorical center of $\HH(A)$ is just $k$, in degree~0.
\end{proof}

\begin{exam}
  Let $A$ be any finite dimensional unipotent Hopf algebra
  and $\K=\Thick({_\phi A_1} \mid \phi \in H)$ for some finite subgroup $H$ of $\Aut(A)$. 
 By the description of its spectrum in Theorem \ref{thm:spclrp}, the interpretation involving Hochschild cohomology for $\Spc(\E)$ given in Theorem \ref{thm:nvy-conj-gives-hochschild} (contingent on finite generation and \cite[Conjecture E]{NVY3}), and the description of the categorical center given by \cite[Corollary 4.2.2]{NVY3}, it follows that $\Spc(\K)$ is homeomorphic to $\Proj$ of the categorical center of the Hochschild cohomology ring for this category, taken with respect to the given thick generators. In particular, the comparison map from Section \ref{subsect:catcenter} is a homeomorphism. Comparing Proposition \ref{prop:elemab-catcent} with Proposition \ref{prop:ulrp-elemab} illustrates the extent to which we can currently connect these two objects for the entire $\ulrp(A^{\env})$, when $A$ is an elementary abelian $p$-group.
\end{exam}


\bibliography{bibl}
\bibliographystyle{alpha}
\end{document}